\documentclass[12pt]{amsart}
\usepackage[margin=1.15in]{geometry}
\numberwithin{equation}{section}

\usepackage{amssymb}
\usepackage{enumerate, xspace}
\usepackage{marvosym} 
\usepackage[sans]{dsfont}        
\usepackage{amsthm}
\usepackage{etoolbox}
\usepackage[bottom,first]{draftcopy}
\usepackage{changebar}
\usepackage[backgroundcolor=blue!20!white, linecolor=blue!20!white,textsize=footnotesize]{todonotes}
\usepackage[colorlinks]{hyperref}
\usepackage{graphicx}
\usepackage{mathtools}
\usepackage{cleveref}

\usepackage{comment}
\allowdisplaybreaks[4]
\newtheorem{thm}{Theorem}[section]
\newtheorem{lem}[thm]{Lemma}
\newtheorem{cor}[thm]{Corollary}

\newtheorem{prop}[thm]{Proposition}

\newtheorem{defin}{Definition}

\newtheorem{rem}{Remark}[section]
\AfterEndEnvironment{rem}{\noindent\ignorespaces}

\newcommand\cE{{\mathcal E}}
\newcommand\cF{{\mathcal F}}

\newcommand\cL{{\mathcal L}}
\newcommand\cO{{\mathcal O}}
\newcommand\cM{{\mathcal M}}
\newcommand\cN{{\mathcal N}}

\newcommand\fl{{\mathfrak{l}}}

\newtheorem{innercustomthm}{Theorem}
\newenvironment{customthm}[1]
  {\renewcommand\theinnercustomthm{#1}\innercustomthm}
  {\endinnercustomthm}

\newcommand\Ban{{\mathbb{B}}}

\newcommand\ve{\varepsilon}
\newcommand\eps{\epsilon}
\newcommand\vf{\varphi}

\newcommand{\wh}[1]{\widehat{#1}}

\def\eps{{\varepsilon}}

\def\Prob{{\mathbb{P}}}

\def\EXP{{\mathbb{E}}}

\def\complex{\mathbb{C}}

\def\naturals{\mathbb{N}}

\def\reals{\mathbb{R}}

\def\integers{\mathbb{Z}}

\def\P{{\partial }}

\def\bc{\mathbf{c}}

\def\bh{\mathbf{h}}

\def\cA{\mathcal{A}}

\def\cJ{\mathcal{J}}

\def\cF{\mathcal{F}}

\def\cE{\mathcal{E}}

\def\cK{\mathcal{K}}

\def\cL{\mathcal{L}}

\def\cM{\mathcal{M}}

\def\cN{\mathcal{N}}

\def\cO{\mathcal{O}}

\def\fN{\mathfrak{N}}

\def\fn{\mathfrak{n}}

\makeatother

\def\beq{\begin{equation}}
\def\eeq{\end{equation}}

\newcommand{\BV}{{\operatorname{BV}}}

\begin{document}

\title[Edgeworth Expansions]{Edgeworth expansions for weakly dependent random variables}
\author{Kasun Fernando and Carlangelo Liverani}

\address{Kasun Fernando\\
Department of Mathematics\\
University of Maryland \\
4176 Campus Drive\\
College Park, MD 20742-4015, United States.}
\email{{\tt abkf@math.umd.edu}}

\address{Carlangelo Liverani\\
Dipartimento di Matematica\\
II Universit\`{a} di Roma (Tor Vergata)\\
Via della Ricerca Scientifica, 00133 Roma, Italy.}
\email{{\tt liverani@mat.uniroma2.it}}

\begin{abstract}
We discuss sufficient conditions that guarantee the existence of asymptotic expansions for the CLT for \textit{weakly dependent} random variables including observations arising from sufficiently chaotic dynamical systems like piece-wise expanding maps, and strongly ergodic Markov chains. We primarily use spectral techniques to obtain the results. 
\end{abstract}
\maketitle

\section{Introduction}\label{Intro}
Let $S_N=\sum_{n=1}^N X_n$ be a sum of \textit{weakly dependent} random variables. We say that 
$S_N$ satisfies the Central Limit Theorem if there are real constants $A$ and $\sigma>0$ such that
\begin{equation}\label{CLTConv}
\lim_{N\to\infty} \Prob\left(\frac{S_N-NA}{\sqrt{N}}\leq z\right)=\fN(z) 
\end{equation}
where $$\fN(z)=\int_{-\infty}^z \fn (y) dy \text{ and } \fn(y)=\frac{1}{\sqrt{2\pi\sigma^2} } e^{-\frac{y^2}{2\sigma^2}}. $$
An important problem is to estimate the rate of convergence of \eqref{CLTConv}.

To this end, an asymptotic expansion, now commonly referred to as the Edgeworth
expansion, was formally derived by Chebyshev in 1859. 

\begin{defin} \label{EdgeExpDef}
$S_N$ admits Edgeworth expansion of order $r$ if there are polynomials 
$P_1(z),\dots, P_r(z)$ such that
$$ \Prob\left(\frac{S_N-NA}{\sqrt{N}}\leq z\right)-\fN(z)=\sum_{p=1}^r \frac{P_p(z)}{N^{p/2}} \fn(z)+
o\left(N^{-r/2}\right)
 $$
uniformly for $z \in \reals$. 
\end{defin}
\begin{rem}\label{Uniq1}
It is an easy observation that order $r$ Edgeworth expansion of $S_N$, if it exists, is unique. Suppose $\{P_p(z)\}_p$ and $\{\tilde{P}_p(z)\}_p$, $1 \leq p \leq r$ are polynomials corresponding to two Edgeworth expansions. Then,
$$
\sum_{p=1}^r \frac{P_p(z)}{N^{p/2}} \fn(z) = \sum_{p=1}^r \frac{\tilde{P}_p(z)}{N^{p/2}} \fn(z)+ o\left(N^{-r/2}\right).
$$
Multiplying by $\sqrt{N}$ taking the limit $N \to \infty$ we have $P_1(z)=\tilde{P}_1(z)$. Therefore, 
$$
\sum_{p=2}^r \frac{P_p(z)}{N^{p/2}} \fn(z) = \sum_{p=2}^r \frac{\tilde{P}_p(z)}{N^{p/2}} \fn(z)+ o\left(N^{-r/2}\right)
$$
Then, multiplying by $N$ and taking $N \to \infty$, $P_2(z)=\tilde{P}_2(z)$. Continuing this $r$ times we can conclude $P_p(z) = \tilde{P}_p(z)$ for $1\leq p \leq r$.
\end{rem}

When $X_i$'s are independent and identically distributed (i.i.d.), it is known that the order $1$ Edgeworth expansion exists if and only if the distribution of $X$ is non-lattice (see \cite{ES}). Therefore the following asymptotic expansion for the Local Central Limit Theorem (LCLT) for lattice random variables is also useful.
\begin{defin} \label{LatticeEdgeExpDef}
Suppose that $X_n$'s are integer valued. We say that $S_N$ admits a lattice Edgeworth expansion of order $r$, if there are polynomials $P_{0,d}, \dots, P_{r,d}$ and a number $A$ such that
$$ \sqrt{N}\Prob(S_N=k)=\fn\left(\frac{k-NA}{\sqrt{N}}\right) 
\sum_{p=0}^r \frac{P_{p,d}((k-NA)/\sqrt{N})}{N^{p/2}}+o\left(N^{-r/2}\right) $$
uniformly for $k \in \integers$. 
\end{defin}

\begin{rem}
As in \cref{Uniq1}, we can prove the uniqueness of this expansion. Because $P_{p,d}$'s have finite degree, say at most $q$, choose $N$ large enough so that $S_N$ has more than $q$ values. Then the argument in \cref{Uniq1} applies mutatis mutandis. 
\end{rem}

During the 20th century, the work of Lyapunov, Edgeworth, Cram\'er, Kolmogorov, Ess\'een, Petrov, Bhattacharya and many others led to the development of the theory of asymptotic expansions of these two forms. See \cite{Ha, IL} and references therein, for more details.

In \cite{Br}, weak (or functional) forms of Edgeworth expansions are introduced. These expansions yield the asymptotics of $\EXP(f(S_N))$. 

Let $(\cF,\|\cdot\|)$ be a function space.

\begin{defin} \label{WGEdgeExpDef}
$S_N$ admits weak global Edgeworth expansion of order $r$ for $f \in \cF$, if there are polynomials $P_{0,g}(z),\dots P_{r,g}(z)$ and $A$ $($which are independent of $f)$ such that 
$$ \EXP(f(S_N-NA))= \sum_{p=0}^r \frac{1} {N^{\frac{p}{2}}} \int P_{p,g}(z) \fn(z) 
f\big(z\sqrt{N}\big) dz+\|f\|\hspace{2pt}o\left(N^{-(r+1)/2}\right).$$
\end{defin}

\begin{defin} \label{WLEdgeExpDef}
$S_N$ admits weak local Edgeworth expansion of order $r$ $f \in \cF$, if there are polynomials $P_{0,l}(z),\dots P_{r,l}(z)$ and $A$ $($which are independent of $f)$ such that 
$$ \sqrt{N}\EXP(f(S_N-NA))=\frac{1}{2\pi} \sum_{p=0}^{\lfloor r/2 \rfloor} \frac{1}{N^{p}}\int  P_{p,l}(z) f(z) dz+\|f\|\hspace{2pt}o\left(N^{-r/2}\right).
 $$
\end{defin}

We also introduce the following asymptotic expansion which yields an averaged form of the error of approximation. 
\begin{defin}\label{AveEdgeExpDef}
$S_N$ admits averaged Edgeworth expansion of order $r$ if there are polynomials 
$P_{1,a}(z),\dots P_{r,a}(z)$ and numbers $k,m$ such that for $f\in \cF$ we have
\begin{multline*}
\int \left[\Prob\left(\frac{S_N-NA}{\sqrt{N}}\leq z+\frac{y}{\sqrt{N}}\right)-\fN\left(z+\frac{y}{\sqrt{N}}\right)\right] f(y) dy \\ = \sum_{p=1}^r \frac{1} {N^{p/2}} \int P_{p,a}\left(z+\frac{y}{\sqrt{N}}\right) 
\fn\left(z+\frac{y}{\sqrt{N}}\right) 
f\left(y\right) dy+\|f\|\hspace{2pt}
o\left(N^{-r/2}\right).
\end{multline*}
\end{defin}

\begin{rem}
All of these weak forms of expansions are unique provided that $\cF$ is dense in $C^\infty_c$ with respect to $\|\cdot\|_\infty$. If there are two different weak global expansions with polynomials $\{P_{p,g}\}$ and $\{\tilde{P}_{p,g}\}$, the argument in \cref{Uniq1} yields, $$\int P_{p,g}(z) \fn(z) 
f\big(z\sqrt{N}\big) dz = \int \tilde{P}_{p,g}(z) \fn(z) 
f\big(z\sqrt{N}\big) dz$$
for all $f \in C^\infty_c$ which gives us the equality, $P_{p,g}(z)=\tilde{P}_{p,g}(z)$. The same idea works for the other two expansions. 
\end{rem}

We have seen that these asymptotic expansions are unique. They also form a hierarchy. Because the dependencies among the expansions are independent of the abstract setting we introduce in \cref{results}, we postpone the discussion about this hierarchy to \Cref{appen}. Due to this hierarchy, in the absence of one, others can be useful in extracting information about the rate of convergence in \eqref{CLTConv}.  

Previous results on existence of Edgeworth expansions (see \cite{Feller2}) assumes independence of random variables $X_n$.  For many applications the independence assumption of random variables is too restrictive. Because of this reason there have been attempts to develop a theory of Edgeworth expansions for weakly dependent random variables where weak dependence often refers to asymptotic decorrelation. See \cite{CP, GH, HP, NG1, NG2} for such examples. The primary focus of these is the classical Edgeworth expansions introduced in \Cref{EdgeExpDef} and \Cref{LatticeEdgeExpDef}. 

Except in \cite{CP}, the sequences of random variables considered are uniformly ergodic Markov processes with strong recurrent properties or processes approximated by such Markov processes. In \cite{CP}, the authors consider aperiodic subshifts of finite type endowed with a stationary equilibrium state and give explicit construction of the order 1 Edgeworth expansion. They also prove the existence of higher order classical Edgeworth expansions under a rapid decay assumption on the tail of the characteristic function. 

The goal of our paper is to generalize these results and to provide suffcient conditions that guarantee the existence of Edgeworth expansions for weakly dependent random variables including observations arising from sufficiently chaotic dynamical systems, and strongly ergodic Markov chains. In fact, we introduce a widely applicable theory for both classical and weak forms of Edgeworth expansions and significantly improve pre-existing results. In \cref{results}, we discuss these in detail.

The paper is organized as follows. In \cref{results}, we introduce the abstract setting we work on and state the main results on existence of Edgeworth expansions. In \cref{proofs}, we prove these results by constructing the Edgeworth polynomials using the characteristic function and concluding that they satisfy the specific asymptotics. In \cref{Coeff}, we relate the coefficients of these polynomials to moments of $S_N$ and provide an algorithm to compute coefficients. A few applications of the Edgeworth expansions such as Local Limit Theorems and Moderate Deviations, are discussed in \cref{App}. In the last section, we give examples of sequences of random variables for which our theory can be applied. These include observations arising from piecewise expanding maps of an interval, markov chains with finitely many states and markov processes which are strongly ergodic. 
\section{Main results.}\label{results}
We assume that there is a Banach space $\Ban$, a family of bounded linear operators $\cL_t:\Ban\to\Ban$ and vectors $v\in \Ban, \ell \in \Ban'$ such that
\begin{equation}
\label{MainAssum}
 \EXP\left(e^{it S_N}\right)=\ell(\cL^N_t v),\ t\in \reals. 
\end{equation}

We will make the following assumptions on the family $\cL_t.$ \vskip2mm
\begin{itemize} \setlength\itemsep{9pt}
\item[(A1)] $t \mapsto \cL_t$ is continuous and there exists $s \in \naturals$ and $\delta > 0$ such that $t\mapsto \cL_t$ is $s$ times continuously differentiable for $|t| \leq\delta$.
\item[(A2)] 1 is an isolated and simple eigenvalue of $\cL_0$, all other eigenvalues of $\cL_0$ have absolute value less than $1$ and its essential spectrum is contained strictly inside the disk of radius $1$ (spectral gap).
\item[(A3)] For all $t\neq 0$, 
sp$(\cL_t)\subset \{|z|<1\}$.
\item[(A4)] There are positive real numbers $K, r_1,r_2$ and $N_0$ such that $\left\Vert \cL_t^N \right\Vert \leq \frac{1}{N^{r_2}}$ for all $t$ satisfying $K\leq |t| \leq N^{r_1}$ and $N>N_0$.
\end{itemize}
\begin{rem}\label{OnA1-A4}\
\begin{itemize}
\item[$1.$] In practice we check $(A3)$ by showing that when $t \neq 0$, the spectral radius of $\cL_t$ is at most $1$ and $\cL_t$ does not have an eigenvalue on the unit circle. Because the spectrum of a linear operator is a closed set this would imply that {\upshape sp}$(\cL_t)$ is contained in a closed disk strictly inside the unit disk. \item[$2.$] Suppose $(A4)$ holds. Let $N_1>N_0$ be such that $N^{(r_1-\epsilon)/r_1}_1>N_0$. Then, for all $N>N_1,$
\begin{align*}
\phantom{aaaaa}\|\cL^N_t\|&\leq \|(\cL^{\lceil N^{(r_1-\epsilon)/r_1} \rceil }_t)^{N^{\epsilon/r_1}_1}\| \leq \|(\cL^{\lceil N^{(r_1-\epsilon)/r_1} \rceil }_t)\|^{N^{\epsilon/r_1}_1} \\ &\leq \frac{1}{\lceil N^{(r_1-\epsilon)/r_1}  \rceil^{r_2N^{\epsilon/r_1}_1}}\ \text{for}\ K \leq |t| \leq N^{r_1-\epsilon}\\ &\leq \frac{1}{N^{r_2K_{N_1}}}
\end{align*}
where $K_{N_1}=\frac{r_1-\epsilon}{r_1}N^{\epsilon/r_1}$. 
Therefore fixing $N_1$ large enough we can make $r_2K_{N_1}$ as large as we want. Hence, given $(A4)$, by slightly decreasing $r_1$, we may assume $r_2$ is sufficiently large. 
\item[$3.$] Suppose $(A1), (A2)$ and $(A3)$ are satisfied with $s \geq 3$. Then, \cite[Theorem 2.4]{G} implies that there exists $A \in \reals$ and $\sigma^2 \geq 0$ such that  
\begin{equation}\label{CLT}
\phantom{aaaaaa}\frac{S_N-NA}{\sqrt{N}} \xrightarrow{d} \cN(0,\sigma^2).
\end{equation}
Our interest is in $S_N$ that satisfies the CLT i.e.\hspace{3pt}the case $\sigma^2 >0$. In applications we specify conditions which guarantee this. Therefore, in the following theorems we always assume that $\sigma^2 >0$. 
\end{itemize}

\end{rem}

Now we are in a position to state our first result on the existence of the classical Edgeworth expansion for weakly dependent random variables.
 
\begin{thm}\label{EdgeExp}
Let $r\in \mathbb{N}$ with $r \geq 2$. Suppose $(A1)$ through $(A4)$ hold with $s=r+2$ and $r_1> \frac{r-1}{2}$. 
Then $S_N$ admits the Edgeworth expansion of order $r$.
\end{thm} 

Next, we examine the error of the order $1$ Edgeworth expansion in more detail. 
We first show that the order 1 expansion exists if (A1) through (A3) hold with $s=3$. Next, we show that the error of approximation can be improved if (A4) holds.

\begin{thm}\label{O1Exp}
Suppose $(A1)$ through $(A3)$ hold with $s \geq 3$. Then, the order $1$ Edgeworth expansion exists.
\end{thm} 

\begin{thm}\label{O1ExpError}
Suppose $(A1)$ through $(A4)$ hold with $s \geq 4$. Then,
$$ \Prob\left(\frac{S_N-NA}{\sqrt{N}}\leq z\right)=\fN(z)+\frac{P_1(z)}{N^{1/2}} \fn(z)+
\cO\left(\frac{1}{N^s}\right)
 $$
where $s=\min\big\{1,\frac{1}{2}+r_1 \big\}$.
\end{thm}

As one would expect, \Cref{O1ExpError} shows that more precise asymptotics than the usual $o(N^{-\frac{1}{2}})$ can be obtained when the characteristic function has better decay. Its proof shows that the error depends mostly on the expansion of the characteristic function at $0$. This is an indication that the error in \Cref{O1Exp} cannot be improved more than by a factor of $\frac{1}{\sqrt{N}}$ even when $r_1$ is large.

In \cite{CP}, analogous results are obtained for subshifts of finite type in the stationary case and an explicit description of the first order Edgeworth expansion is given. Here, we consider a wider class of (not necessarily stationary) sequences and give explicit descriptions of higher order Edgeworth polynomials by relating the coefficients to asymptotic moments. Also, we improve the condition $$H_r:\ |\EXP(e^{itS_N})|\leq K\bigg(1-\frac{c}{|t|^\alpha}\bigg)^n,\ \frac{\alpha(r-1)}{2}<1,\ |t|>K$$
found in \cite{CP} by replacing it with (A4). In addition, this allows us to obtain better asymptotics for the first order expansion. 

We also extend the results on the existence of weak Edgeworth expansions for i.i.d.\hspace{3pt}random variables found in \cite{Br}. In \cref{IID}, we compare our results with the earlier ones. 

Before we mention these results, we define the space $F_k^m$ of functions. Put $$C^{m}(f)=\max_{0\leq j \leq m} \|f^{(j)}\|_{\text{L}^1}\ \ \text{and} \ \ C_k(f)=\max_{0\leq j \leq k} \|x^jf\|_{\text{L}^1}.$$ Define $$C^m_k(f)=C^m(f)+C_k(f).$$ We say $f \in F_k^m$ if $f$ is $m$ times continuously differentiable and $C^m_k(f) < \infty$.

\begin{thm}\label{WLEdgeExp}
Suppose $(A1)$ through $(A4)$ hold with $s=r+2$. 
Choose $q\in \naturals$ such that $q > \frac{r+1}{2r_1}$. Then, for $f\in F_{r+1}^{q+2}$, $S_N$ admits the weak local Edgeworth expansion of order $r$. 
\end{thm}

\begin{thm}\label{WGEdgeExp}
Suppose $(A1)$ through $(A4)$ hold with $s=r+2$. 
Choose $q\in \naturals$ such that $q > \frac{r+1}{2r_1}$. Then, for $f\in F_{0}^{q+2}$, $S_N$ admits the weak global Edgeworth expansion of order $r$. 
\end{thm}

In \Cref{WLEdgeExp} and \Cref{WGEdgeExp}, $f$ is required to have at least three derivatives in order to guarantee the integrability of Fourier transforms of $f$ and its derivatives. In addition to (A1) through (A4), if we have, 
\begin{itemize} \setlength\itemsep{9pt}
\item[(A5)] There exists $\alpha>0$ and $N_1$ such that $\|\cL^N_t\| \leq  \frac{C}{t^\alpha}$ for $|t| > N^{r_1}$ for $N>N_1$. 
\end{itemize}
then we can improve this assumption to $f$ having only one continuous derivative. 

\begin{customthm}{2.2*}\label{NewWLEdgeExp}
Suppose $(A1)$ through $(A5)$ hold with $s=r+2$ and $\alpha > \frac{r+1}{2r_1}$ for sufficiently large $N$. Then, for $f\in F_{r+1}^{1}$, $S_N$ admits the weak local Edgeworth expansion of order $r$. 
\end{customthm}

\begin{customthm}{2.3*}\label{NewWGEdgeExp}
Suppose $(A1)$ through $(A5)$ hold with $s=r+2$ and $\alpha > \frac{r+1}{2r_1}$ for sufficiently large $N$. Then, for $f\in F_{0}^{1}$, $S_N$ admits the weak global Edgeworth expansion of order $r$. 
\end{customthm}

The proofs of these theorems are minor modifications of the proofs of the previous two theorems. This is described in \cref{AltPrfWEdge}. 

The next theorem gives sufficient conditions for the existence of the averaged Edgeworth expansion. 

\begin{thm}\label{AVGEdgeExp}
Suppose $(A1)$ through $(A4)$ hold with $s=r+2$. Choose $q \in \naturals$ such that 
$q> \frac{r}{2r_1}$. Then, $S_N$ admits the averaged Edgeworth expansion of order $r$ for $f\in F^{q}_0$. 
\end{thm}

We note that for integer valued random variable assumptions (A3) and (A4) cannot hold since the characteristic function of $S_N$ is $2\pi$-periodic. Therefore we replace (A3) by,
\begin{itemize}
\item[$\widetilde{(\text{A3})}$] When $t\not\in2\pi \integers$, sp$(\cL_t)\subset \{|z|<1\}$ and when $t \in 2\pi \integers$, sp$(\cL_t)\subset \{|z|< 1\} \cup \{1\}$. 
\end{itemize}
Also, because of periodicity of the characteristic function, an assumption similar to (A4) is not required. 

The following theorem provides conditions for the existence of asymptotic expansions for the LCLT for weakly dependent integer valued random variables. A similar result for $X_n$'s that are $\integers^d$-valued, is obtained in \cite{PN}. Compare with Proposition 4.2 and 4.4 therein.  

\begin{thm}\label{LatticeEdgeExp}
Suppose $X_n$ are integer valued, $(A1),(A2)$ and $\widetilde{(A3)}$ are satisfied with $s=r+2$. Then $S_N$ admits the order $r$ lattice Edgeworth expansion.  
\end{thm}

\section{Proofs of the main results}\label{proofs}
Here we prove the main results of the paper. From now on we work in the setting described in \cref{results}. 
\begin{proof}[Proof of Theorem \ref{EdgeExp}]  
We seek polynomials $P_p(x)$ with real coefficients such that
\begin{equation}\label{EdgeExpP}
\Prob\left(\frac{S_n-nA}{\sqrt{n}}\leq x\right)-\fN(x)=\sum_{p=1}^r \frac{P_p(x)}{n^{p/2}} \fn(x)+
o\left(n^{-r/2}\right).
\end{equation} 
Once we have found suitable candidates for $P_p(x)$ we can apply the Berry-Esseen inequality,
\begin{equation}\label{eq:basic}
\left|F_{n}(x)-\cE_{r,N}(x)\right|\le \frac 1\pi\int_{-T}^{T}\left|\frac{\widehat F_{n}(t)-\widehat \cE_{r,n}(t)}{t}\right|dt+\frac{C_0}{T},
\end{equation}
where $$F_n(x)= \Prob\left(\frac{S_n-nA}{\sqrt{n}}\leq x\right), \ \ \ \cE_{r,n}(x)=\fN(x)+\sum_{p=1}^r \frac{P_p(x)}{n^{p/2}} \fn(x),$$ and $C_0$ is independent of $T$. We refer the reader to \cite[Chapter XVI.3]{Feller2} for a proof of \eqref{eq:basic}. What follows is a formal derivation of $P_p(x)$. Later, we will use \eqref{eq:basic} along with other estimates to prove \eqref{EdgeExpP}. 

It follows from (A1), (A2) and classical perturbation theory (see \cite[IV.3.6 and VII.1.8]{Kato}) that there exist $\delta>0$ such that for $|t|\leq\delta$, $\cL_t$ has a top eigenvalue $\mu(t)$ which is simple and the remainder of the spectrum is contained in a strictly smaller disk.
One can express $\cL_t$ as 
\begin{equation}\label{OpDecom}
\cL_{t}=\mu(t)\Pi_{t}+\Lambda_t
\end{equation}
where $\Pi_t$ is the eigenprojection to the top eigenspace of $\cL_t$ and $\Lambda_t=(I-\Pi_t)\cL_t$. Because $\Lambda_t\Pi_t = \Pi_t\Lambda_t = 0$, iterating \eqref{OpDecom}, we obtain $$\cL^n_{t}=\mu^n(t)\Pi_{t}+\Lambda^n_t.$$ 
Using (A3) and compactness, there exist $C$ (which does not depend on $n$ and $t$) and $0<r<1$ such that $\|\Lambda^n_t\|\leq Cr^n$ for all $|t|\leq\delta$. By \eqref{MainAssum}, 
\begin{equation}\label{eq:char fn}
\EXP(e^{it S_n/\sqrt{n}})=\mu\Big(\frac{t}{\sqrt{n}}\Big)^n \ell\big( \Pi_{t/\sqrt{n}} v \big) + \ell\big(\Lambda^n_{t/\sqrt{n}} v\big).
\end{equation}
Now, we focus on the first term of \eqref{eq:char fn}. Put
\begin{equation}\label{av proj}
Z(t)=\ell( \Pi_{t} v).
\end{equation}
Then, substituting $t=0$ in \eqref{eq:char fn} yields $1=Z(0)+\ell(\Lambda^n_0 v)$. Also, we know that $\lim_{n \to \infty} \|\Lambda_0^n v\|=0$. This gives $\lim_{n \to \infty} \ell(\Lambda^n_0 v) =0$. Therefore, $Z(0)=1$ and $Z(t) \neq 0$ when $|t|<\delta$. Also, this shows that $\ell(\Lambda^n_0 v)=0$ for all $n$. Next, note that $t \mapsto \mu(t)$ and $t \mapsto \Pi_t$ are $r+2$ times continuously differentiable on $|t|<\delta$ (see \cite[IV.3.6 and VII.1.8]{Kato}). Therefore, $Z(t)$ is $r+2$ times continuously differentiable on $|t|<\delta$. 

Now we are in a position to compute $P_p(x)$. To this end we make use of ideas in \cite[Chapter XVI]{Feller2} (where the Edgeworth expansions for i.i.d.\hspace{3pt}random variables are constructed) and \cite{G} (where the CLT is proved using Nagaev-Guivarc'h method). 

Consider the function $\psi$ such that,
\begin{align*}
\log \mu\Big(\frac{t}{\sqrt{n}}\Big) = \frac{iAt}{\sqrt{n}}-\frac{\sigma^2t^2}{2n} + \psi\Big(\frac{t}{\sqrt{n}}\Big) \iff \mu^n\Big(\frac{t}{\sqrt{n}}\Big)=e^{\frac{inAt}{\sqrt{n}}-\frac{\sigma^2t^2}{2}}\exp \Big( n\psi\Big(\frac{t}{\sqrt{n}}\Big)\Big).
\end{align*}
where $A=\underset{n \to \infty}{\lim} \EXP\big(\frac{S_n}{n}\big)$ is the asymptotic mean and $\sigma^2=\underset{n \to \infty}{\lim} \EXP\big(\big[\frac{S_n-nA}{\sqrt{n}}\big]^2\big)$ is the asymptotic variance. (For details see \cref{Coeff}.)

By \eqref{eq:char fn} we have,
\begin{equation}\label{CharFn1}
\EXP(e^{it \frac{S_n-nA}{\sqrt{n}}})=e^{-\frac{\sigma^2t^2}{2}} \exp \Big( n\psi\Big(\frac{t}{\sqrt{n}}\Big) \Big)Z\Big(\frac{t}{\sqrt{n}}\Big) + e^{-\frac{inAt}{\sqrt{n}}}\ell\Big(\Lambda^n_{\frac{t}{\sqrt{n}}} v\Big)
\end{equation}

Notice that $\psi(0)=\psi'(0)=0$ and $\psi(t)$ is $r+2$ times continuously differentiable. Now, denote by $t^2\psi_r(t)$ the order $(r+2)$ Taylor approximation of $\psi$. Then, $\psi_r$ is the unique polynomial such that $\psi(t)=t^2\psi_r(t)+o(|t|^{r+2})$. Also, $\psi_r(0)=0$ and $\psi_r$ is a polynomial of degree $r$. In fact, we can write $\psi(t)=t^2\psi_r(t)+t^{r+2}\tilde{\psi}_r(t)$ where $\tilde{\psi}_r$ is continuous and $\tilde{\psi}_r(0)=0$. Thus,
\begin{align*}
\exp \Big( n\psi\left(\frac{t}{\sqrt{n}}\Big) \right)=\exp\Big(t^2 \psi_r\Big(\frac{t}{\sqrt{n}}\Big)+\frac{1}{n^{r/2}}t^{r+2}\tilde{\psi}_r\Big(\frac{t}{\sqrt{n}}\Big)\Big).
\end{align*}
Denote by $Z_r(t)$ the order$-r$ Taylor expansion of $Z(t)-1$. Then, $Z_r(0)=0$ and $Z(t)=1+Z_r(t)+t^{r}\tilde{Z}_r(t)$ with twice continuously differentiable $\tilde{Z}_r(t)$ such that $\tilde{Z}_r(0)=0$. Then, to make the order $n^{-j/2}$ terms explicit, we compute:
\begin{align}\label{PolyComp}
e^{\frac{\sigma^2t^2}{2}}&\mu^n\Big(\frac{t}{\sqrt{n}}\Big)Z\Big(\frac{t}{\sqrt{n}}\Big) \nonumber \\ &= e^{\frac{\sigma^2t^2}{2}}\mu^n\Big(\frac{t}{\sqrt{n}}\Big)\exp \log Z\Big(\frac{t}{\sqrt{n}}\Big) \nonumber \\ &=\exp\Big(t^2 \psi_r\Big(\frac{t}{\sqrt{n}}\Big)+\frac{1}{n^{r/2}}t^{r+2}\tilde{\psi}_r\Big(\frac{t}{\sqrt{n}}\Big) \nonumber \\ &\phantom{=\exp\Big(t^2 \psi_r\Big(\frac{t}{\sqrt{n}}\Big)+\frac{1}{n^{r/2}}t^{r+2}}-\sum_{k=1}^{r}\frac{(-1)^{k+1}}{k}\Big[Z_r\Big(\frac{t}{\sqrt{n}}\Big)\Big]^k -\frac{1}{n^{r/2}}t^r\overline{Z}_r\Big(\frac{t}{\sqrt{n}}\Big)\Big)  \nonumber \\&=1+ \sum_{m=1}^r \frac{1}{m!} \Big[t^2 \psi_r\Big(\frac{t}{\sqrt{n}}\Big)-\sum_{k=1}^{r}\frac{(-1)^{k+1}}{k}\Big[Z_r\Big(\frac{t}{\sqrt{n}}\Big)\Big]^k\Big]^m  \nonumber \\ \nonumber &\phantom{=\exp\Big(t^2 \psi_r\Big(\frac{t}{\sqrt{n}}\Big)++} + \frac{1}{n^{r/2}}t^{r+2}\tilde{\psi}_r\Big(\frac{t}{\sqrt{n}}\Big) -\frac{1}{n^{r/2}}t^r\overline{Z}_r\Big(\frac{t}{\sqrt{n}}\Big)+t^{r+1}\cO\big(n^{-\frac{r+1}{2}}\big)  \nonumber \\ &= \sum_{k=0}^{r} \frac{A_k(t)}{n^{k/2}}+\frac{t^{r}}{n^{r/2}}\varphi\Big(\frac{t}{\sqrt{n}}\Big)+t^{r+1}\cO\big(n^{-\frac{r+1}{2}}\big)
\end{align}
where $A_0 \equiv 1$, $\varphi(t)=t^2\tilde{\psi}_r(t)-\overline{Z}_r(t)$ is continuous and $\varphi(0)=0$. Here $\overline{Z}_r$ is the remainder of $\log Z(t)$ when approximated by powers of $Z_r$.  Next write,
\begin{align}\label{MainPolyExp}
Q_n(t)=\sum_{k=1}^{r} \frac{A_k(t)}{n^{k/2}}.
\end{align}
Notice that 
\begin{equation}\label{Parity}
A_k\ \text{and}\ k\ \text{have the same parity.} 
\end{equation}
This can be seen directly from the construction, because we collect terms with the same power of $n^{-1/2}$, $\psi_r$ and $Z_r$ are a polynomial in $\frac{t}{\sqrt{n}}$ with no constant term and we take powers of $t^2\psi_r(t)$ and $Z_r(t)$, the resulting $A_k$ will contain terms of the form $c_st^{2s+k}$. 

We claim that,
\begin{align}\label{Near0Est}
\int_{|t|<\delta \sqrt{n}}&\bigg|\frac{\mu^n\big(\frac{t}{\sqrt{n}}\big)Z\big(\frac{t}{\sqrt{n}}\big)-e^{-\frac{t^2\sigma^2}{2}} - e^{-\frac{t^2\sigma^2}{2}}Q_n(t)}{t} \bigg|\, dt \\ &= \int_{|t|<\delta \sqrt{n}}e^{-\frac{t^2\sigma^2}{2}}\bigg|\frac{\exp\big[ n \psi\big(\frac{t}{\sqrt{n}}\big)+\log Z\big(\frac{t}{\sqrt{n}}\big)\big]-1 - Q_n(t)}{t} \bigg|\, dt \nonumber  \\  &= o\left(n^{-r/2}\right).  \nonumber 
\end{align}
We note that from the choice of $Q_n$,
\begin{align*}
\frac{\exp\left[ n \psi\big(\frac{t}{\sqrt{n}}\big)+\log Z\big(\frac{t}{\sqrt{n}}\big)\right]- 1 - Q_n(t)}{t} = \frac{1}{n^{r/2}}\Big(t^{r-1}\varphi\Big(\frac{t}{\sqrt{n}}\Big)+t^r\cO\big(n^{-\frac{r+1}{2}}\big) \Big)
\end{align*}
where $\varphi(t)=o(1)$ as $t \to 0$. As a result, for all $\ve>0$ the integrand of \eqref{Near0Est} can be made smaller than $\frac{\ve}{n^{r/2}}(t^{r-1}+t^r)e^{-\frac{t^2\sigma^2}{2}}$  by choosing $\delta$ small enough. This proves the claim. 

Even though the following derivation is only valid for $|t|<\delta \sqrt{n}$, once the polynomial function $Q_n(t)$ is obtained as above, we can consider it to be defined for all $t \in \reals$.

Suppose $|t|\leq \delta$. From classical perturbation theory (see \cite[Chapter IV]{Kato}  and \cite[Section 7]{HP}) we have 
\begin{equation}\label{ResidualProj}
\Lambda^n_t = \frac{1}{2\pi i} \int_{\Gamma} z^n(z-\cL_t)^{-1} \, dz
\end{equation}
where $\Gamma$ is the positively oriented circle centered at $z=0$ with radius $\ve_0$. Here $\ve_0$ is uniform in $t$ and $0<\ve_0<1$. 
Now, \begin{align*}
\Lambda^n_t - \Lambda^n_0 &= \frac{1}{2\pi i} \int_{\Gamma} z^n[(z-\cL_t)^{-1}-(z-\cL_t)^{-1}] \, dz \\ &=\frac{1}{2\pi i} \int_{\Gamma} z^n[(z-\cL_0)^{-1}(\cL_t-\cL_0)(z-\cL_t)^{-1}] \, dz.
\end{align*}
Because $\cL_t-\cL_0 = \cO(|t|)$ we have that $\frac{\Lambda^n_t - \Lambda^n_0}{|t|} = \cO(\ve_0^n)$. $\ell \in \Ban'$ and $\ell(\Lambda^n_0v) = 0$ implies that
\begin{align*}
\int_{|t|<\delta \sqrt{n}}\bigg|\frac{e^{-\frac{inAt}{\sqrt{n}}}\ell(\Lambda^n_{t/\sqrt{n}} v)}{t}\bigg| \, dt &=\int_{|t|<\delta \sqrt{n}}\bigg|\frac{e^{-\frac{inAt}{\sqrt{n}}}\ell(\Lambda^n_{t/\sqrt{n}} v-\Lambda^n_0v)}{t}\bigg| \, dt  \\ &\leq  C\int_{|t|<\delta}\left|\frac{\Lambda^n_{t} - \Lambda^n_0}{t}\right| \, dt = \cO(\ve_0^n).
\end{align*}
This decays exponentially fast to $0$ as $n \to \infty$. This allows us to control the second term in the RHS of \eqref{eq:char fn}. Combining this with \eqref{Near0Est} we can conclude that, 
\begin{equation}\label{FullNear0Est}
\int_{|t|<\delta \sqrt{n}} \left| \frac{\EXP(e^{it\frac{S_n-nA}{\sqrt{n}}})-e^{-\frac{t^2\sigma^2}{2}} - e^{-\frac{t^2\sigma^2}{2}}Q_n(t)}{t} \right| \, dt = o(n^{-r/2}).
\end{equation}
Observe that, 
$$(it)^k e^{-\frac{\sigma^2t^2}{2}}=\frac{1}{\sqrt{2\pi\sigma^2}} \widehat{\frac{d^k}{dt^k}e^{-\frac{t^2}{2\sigma^2}}}=\widehat{\frac{d^k}{dt^k} \fn(t)}$$
where $\widehat{f}(x)=\int e^{-itx}f(t) \, dt$ is the Fourier transform of $f$. 
Therefore, 
\begin{equation}\label{PolyForDensity}
R_j(t)\fn(t) =\frac{1}{\sqrt{2\pi\sigma^2}}A_j\left(-i\frac{d}{dt}\right)\Big[e^{-\frac{t^2}{2\sigma^2}} \Big].
\end{equation}
Then, the required $P_p(x)$ for $p \geq 1$, can be found using the relation,
\begin{equation}\label{EdgePolyRel}
\fn(x) R_p(x)=\frac{d}{dx}\Big[\fn(x)P_p(x)\Big].
\end{equation}
For more details, we refer the reader to \cite[Chapter XVI.3,4]{Feller2}. 

Given $\ve >0$, choose $B> \frac{C_0}{\ve}$ where $C_0$ is as in \eqref{eq:basic}. Let $r \in \naturals$. Then we choose polynomials $P_p(x)$ as described above. Then, from \eqref{eq:basic} it follows that,
\begin{align*}
|F_{n}(x)-\cE_{r,n}(x)| &\leq \frac 1\pi\int_{-Bn^{r/2}}^{Bn^{r/2}}\left|\frac{\EXP(e^{it\frac{S_n-nA}{\sqrt{n}}})-e^{-\frac{t^2\sigma^2}{2}}(1+Q_n(t))}{t}\right|dt+\frac{C_0}{Bn^{r/2}} \\ & \leq I_1 + I_2+ I_3 + \frac{\ve}{n^{r/2}}
\end{align*}
where
$$
I_1=\frac 1\pi\int_{|t|<\delta\sqrt{n}}\left|\frac{\EXP(e^{it\frac{S_n-nA}{\sqrt{n}}})-e^{-\frac{t^2\sigma^2}{2}}(1+Q_n(t))}{t}\right|dt 
$$
$$
I_2=\frac 1\pi\int_{\delta\sqrt{n}<|t|<Bn^{r/2}}\left|\frac{\EXP(e^{it S_n/\sqrt{n}})}{t}\right|dt
$$
$$I_3=\frac 1\pi\int_{|t|>\delta\sqrt{n}}e^{-\frac{t^2\sigma^2}{2}}\left|\frac{1+Q_n(t)}{t}\right|dt.$$

From \eqref{Near0Est} we have that $I_1$ is $o(n^{-r/2})$. Because our choice of $\ve >0$ is arbitrary the proof is complete, if $I_2$ and $I_3$ are also $o(n^{-r/2})$. These follow from \eqref{MidEst}, \eqref{LarEst} and \eqref{Gaussian} below. 

It is easy to see that, 
\begin{align}\label{Gaussian}
\int_{|t|>\delta \sqrt{n}}e^{-\frac{t^2\sigma^2}{2}}\left|\frac{1+Q_n(t)}{t} \right|\, dt = \cO(e^{-cn})
\end{align}
for some $c>0$. Thus, we only need to control,
\begin{align*}
I_2 &=\int_{\delta \sqrt{n}<|t|<Bn^{r/2}}\left|\frac{\EXP(e^{itS_n/\sqrt{n}})}{t}\right|\, dt \\ &= \int_{\delta \sqrt{n}<|t|<\overline{\delta}\sqrt{n}}\left|\frac{\EXP(e^{itS_n/\sqrt{n}})}{t}\right|\, dt  + \int_{\overline{\delta}\sqrt{n}<|t|<Bn^{r/2}}\left|\frac{\EXP(e^{itS_n/\sqrt{n}})}{t}\right|\, dt 
\end{align*}
where $\overline{\delta}> \max \{\delta,K\}$ with $K$ as in (A4). 

By (A3) the spectral radius of $\cL_{t}$ has modulus strictly less than $1$. Because $t \mapsto \cL_t$ is continuous, for all $p<q$, there exists $\gamma<1$ and $C>0$, such that $\|\cL_{t}^m\|\leq C\gamma^m $ for all $p \leq |t| \leq q$ for sufficiently large $m$. Then using \eqref{MainAssum} for sufficiently large $n$ we have, 
\begin{align} \label{MidEst}
\int_{\delta \sqrt{n}<|t|<\overline{\delta}\sqrt{n}}\left|\frac{\EXP(e^{itS_n/\sqrt{n}})}{t}\right|\, dt &\leq \frac{1}{\delta \sqrt{n}}\int_{\delta \sqrt{n}<|t|<\overline{\delta}\sqrt{n}}\|\cL^n_{t/\sqrt{n}}\| \, dt \leq  \frac{C\gamma^n}{\sqrt{n}}. 
\end{align}
This shows that the integral converges to $0$ faster than any inverse power of $\sqrt{n}$. Next for sufficiently large $n$,
\begin{align}\label{LarEst}
\int_{\overline{\delta}\sqrt{n}<|t|<Bn^{r/2}}\left|\frac{\EXP(e^{it S_n/\sqrt{n}})}{t}\right|\, dt &\leq  \frac{1}{\overline{\delta}\sqrt{n}} \int_{\overline{\delta}\sqrt{n}<|t|<Bn^{r/2}} |\ell(\cL^n_{t/\sqrt n} v)| \, dt \\ &\leq \frac{2Bn^{r/2}}{\overline{\delta}n^{r_2+1/2}} \|\ell\| \|v\| \nonumber \\ &=Cn^{\frac{r-1}{2}-r_2} = o(n^{-r/2}).\nonumber
\end{align}

The second inequality is due to assumption (A4) i.e.\hspace{2pt}$\|\cL^n_{t/\sqrt{n}}\| \leq \frac{1}{n^{r_2}}$ where $r_2>\frac{r-1}{2}$ (we can assume $r_2>\frac{r-1}{2}$ for large $n$ due to \Cref{OnA1-A4}) and $K \leq \overline{\delta}<\frac{|t|}{\sqrt{n}}<Bn^{\frac{r-1}{2}}\leq n^{r_1}$ for $n \in \naturals$ with $n^{r_1-\frac{r-1}{2}} \geq B$. 
\end{proof}

The proof of \Cref{O1Exp} follows the same idea. We include its proof for completion. 

\begin{proof}[Proof of \Cref{O1Exp}] 
Because (A1) through (A3) hold with $s \geq 3$, we have \eqref{PolyComp} where $\varphi$ is continuous, $\varphi(0)=0$ and $r=1$. Given $\ve>0$, choose $B>\frac{C_0}{\ve}$. Then, 
\begin{align*}
|F_{n}(x)-\cE_{1,n}(x)| &\leq \frac 1\pi\int_{-B\sqrt{n}}^{B\sqrt{n}}\left|\frac{\EXP(e^{it\frac{S_n-nA}{\sqrt{n}}})-e^{-\frac{t^2\sigma^2}{2}}(1+Q_n(t))}{t}\right|dt+\frac{C_0}{B\sqrt{n}} \\ & \leq I_1 + I_2+ I_3 + \frac{\ve}{B\sqrt{n}}.
\end{align*} 
Because, $\varphi(t)=o(1)$ as $t \to 0$ and
\begin{align*}
\frac{\exp\left[ n \psi\big(\frac{t}{\sqrt{n}}\big)+\log Z\big(\frac{t}{\sqrt{n}}\big)\right]- 1 - Q_1(t)}{t} &= \frac{1}{\sqrt{n}}\varphi\Big(\frac{t}{\sqrt{n}}\Big)+t\hspace{2pt}\cO\Big(\frac{1}{n}\Big)
\end{align*}
we have that,
$$I_1=\int_{|t|<\delta \sqrt{n}} \bigg| \frac{\EXP(e^{it\frac{S_n-nA}{\sqrt{n}}})-e^{-\frac{t^2\sigma^2}{2}} - e^{-\frac{t^2\sigma^2}{2}}Q_1(t)}{t} \bigg| \, dt = o(n^{-1/2}).$$
Also, $I_3=\cO(e^{-cn})$. Finally, because of (A3) there is $\gamma<1$ such that,
\begin{align*}
\int_{\delta\sqrt{n}<|t|<B\sqrt{n}}\left|\frac{\EXP(e^{it S_n/\sqrt{n}})}{t}\right|\, dt = \int_{\delta<|t|<B}\left|\frac{\EXP(e^{it S_n})}{t}\right|\, dt \leq C \sup_{\delta\leq |t|\leq B}\|\cL^n_t\| \leq C\gamma^n
\end{align*}
Combining these estimates we have the result. 
\end{proof}

A slight modification of the previous proof gives us the proof of \Cref{O1ExpError}. Higher regularity assumption gives us better asymptotics near $0$ and the assumption on the faster decay of the characteristic function gives us more control in the mid range. 

\begin{proof}[Proof of \Cref{O1ExpError}]
Because (A1) through (A4) hold with $s \geq 4$, we have \eqref{PolyComp} where $\varphi$ is $C^1$, $\varphi(0)=0$ and $r=1$. Then, 
\begin{align*}
|F_{n}(x)-\cE_{1,n}(x)| &\leq \frac 1\pi\int_{-n^{1/2+r_1}}^{n^{1/2+r_1}}\left|\frac{\EXP(e^{it\frac{S_n-nA}{\sqrt{n}}})-e^{-\frac{t^2\sigma^2}{2}}(1+Q_n(t))}{t}\right|dt+\frac{C_0}{n^{1/2+r_1}} \\ & \leq I_1 + I_2+ I_3 + \frac{C_0}{n^{1/2+r_1}}
\end{align*}Because, $\varphi\big(\frac{t}{\sqrt{n}}\big) \sim \frac{t}{\sqrt{n}}$ near $0$ and 
\begin{align*}
\frac{\exp\left[ n \psi\big(\frac{t}{\sqrt{n}}\big)+\log Z\big(\frac{t}{\sqrt{n}}\big)\right]- 1 - Q_1(t)}{t} &= \frac{1}{\sqrt{n}}\varphi\Big(\frac{t}{\sqrt{n}}\Big)+t\hspace{2pt}\cO\Big(\frac{1}{n}\Big)
\end{align*}
we have that,
$$I_1=\int_{|t|<\delta \sqrt{n}} \bigg| \frac{\EXP(e^{it\frac{S_n-nA}{\sqrt{n}}})-e^{-\frac{t^2\sigma^2}{2}} - e^{-\frac{t^2\sigma^2}{2}}Q_1(t)}{t} \bigg| \, dt = \cO\Big(\frac{1}{n}\Big).$$
Also, $I_3=\cO(e^{-cn})$. As before, \eqref{MidEst} holds for $\overline{\delta}>\max\{\delta, K\}$. 

$\|\cL^n_{t}\| \leq \frac{1}{n^{r_2}}$ where $K \leq \overline{\delta}<|t|<n^{r_1}$.
\begin{align*}
\int_{\overline{\delta}\sqrt{n}<|t|<n^{1/2+r_1}}\left|\frac{\EXP(e^{it S_n/\sqrt{n}})}{t}\right|\, dt &= \int_{\overline{\delta}<|t|<n^{r_1}}\left|\frac{\EXP(e^{it S_n})}{t}\right|\, dt \leq C n^{r_1-r_2+\frac{1}{2}}
\end{align*}
Because $r_2$ can be made arbitrarily large by choosing $n$ large enough, $I_2=\cO\big(\frac{1}{n}\big)$.
Therefore, 
$$|F_{n}(x)-\cE_{1,n}(x)|=\cO\Big(\frac{1}{n^s}\Big)$$
where $s=\min\big\{1,\frac{1}{2}+r_1\big\}$
and we have the required conclusion. 
\end{proof}

\begin{rem}
In the proof above, $I_1$ gives the contribution to the error from the expansion of the characteristic function near 0. This dominates when $r_1 \geq \frac{1}{2}$. 
\end{rem}

Weak forms of Edgeworth expansions are discussed in detail in \cite{Br}. We adapt the ideas found in \cite{Br} to our proofs of Theorems \ref{WLEdgeExp} and \ref{WGEdgeExp}. One key difference is the requirement on $f$ to have two more derivatives than required in \cite{Br}. This compensates for the lack of control over the tail of the characteristic function of $S_N$. 
In fact, it is enough to assume $1+\alpha$ more derivatives. But to avoid technicalities we stick to the stronger regularity assumption. In the i.i.d.\hspace{3pt}case, as shown in \cite{Br}, a Diophantine assumption takes care of this. See section \ref{IID} for a detailed discussion of the i.i.d.\hspace{3pt}case.


\begin{proof}[Proof of Theorem \ref{WLEdgeExp}]
Recall that $\wh{f}(t)=\int e^{-itx}f(x) \, dx$ and pick $A$ as in \eqref{CLT}. Then by Plancherel theorem, 
\begin{align}\label{Plancherel}
\EXP(f(S_n-nA))&=\frac{1}{2\pi}\int \wh{f}(t)\EXP(e^{it(S_n-nA)})\, dt \\
\implies \sqrt{n}\EXP(f(S_n-nA))&=\frac{1}{2\pi}\int \wh{f}\left(\frac{t}{\sqrt{n}}\right)\EXP(e^{it\frac{S_n-nA}{\sqrt{n}}})\, dt. \nonumber
\end{align} 

We first estimate RHS away from $0$. Fix small $\delta>0$. (A particular $\delta$ is chosen later.) Notice that for all $\delta \leq |t| \leq K$ (where $K$ as in (A4)), there exists $c_0 \in (0,1)$ such that $\|\cL^n_{t}\|\leq c_0^n$. Thus,
\begin{align*}
\bigg| \int_{\delta<|t|<K} \wh{f}(t)\EXP(e^{it(S_n-nA)}) \, dt \bigg|\leq \int_{\delta<|t|<K} \left|\wh{f}(t)\ell(\cL^n_{t} v)\right| \, dt  \leq C\|f\|_1c^n_0.
\end{align*}
By Remark \ref{OnA1-A4}, for large $n$ we can assume $r_2 > r_1+(r+1)/2$. Therefore,
\begin{align*}
\bigg| \int_{K<|t|<n^{r_1}} \wh{f}(t)\EXP(e^{it(S_n-nA)}) \, dt \bigg| \leq \|f\|_1\|\ell\|  \|v\| \int_{K<|t|<n^{r_1}} \|\cL^n_{t} \|\, dt  &\leq \frac{C\|f\|_1}{n^{r_2-r_1}} \\ &=\|f\|_1o(n^{-(r+1)/2}).
\end{align*}

Because $f \in F^{q+2}_{r+1}$, we have that $t^q\wh{f}(t)=(-i)^q\wh{f^{(q)}}(t)$ and $\wh{f^{(q)}}$ is integrable. In fact, $|\wh{f^{(q)}}(t)| \leq \frac{C}{(1+|t|)^2}$. Note that we are using only the fact that $f$ is $q+2$ times continuously differentiable with integrable derivatives. Therefore, for this to be true, $f \in F^{q+2}_0$ is sufficient. Integrability of $\wh{f^{(q)}}$ along with $q>\frac{r+1}{2r_1}$ implies,
\begin{align}\label{AtInfty}
\bigg| \int_{|t|>n^{r_1}} \wh{f}(t)\EXP(e^{it(S_n-nA)}) \, dt \bigg| \leq \int_{|t|>n^{r_1}} |\wh{f}(t)| \, dt &\leq \int_{|t|>n^{r_1}}  \Big|\frac{\widehat{f^{(q)}}(t)}{t^q}\Big| \, dt \\ &\leq \frac{\|\widehat{f^{(q)}}\|_1}{n^{r_1q}} = \|\widehat{f^{(q)}}\|_1 o (n^{-(r+1)/2}). \nonumber
\end{align}
Therefore,
\begin{equation}\label{t>delta}
\bigg| \int_{|t|>\delta} \wh{f}(t)\EXP(e^{it(S_n-nA)}) \, dt \bigg| = o(n^{-(r+1)/2}).
\end{equation}

From \eqref{CharFn1}, for $ |t| \leq \delta \sqrt{n}$, we have, $$\EXP(e^{it\frac{S_n-nA}{\sqrt{n}}}) = e^{-\frac{\sigma^2t^2}{2}}e^{t^2\cO(\delta)}(1+\cO(\delta))+\cO(\epsilon^n_0).$$ Thus, choosing small $\delta$, for large $n$ when $|t|<\delta \sqrt{n}$ there exist $c, C>0$ such that 
$$\big|\EXP\big(e^{it\frac{S_n-nA}{\sqrt{n}}}\big) \big| \leq Ce^{-ct^2}.$$ 
Then, $$ \sqrt{D \log n} < |t|<\delta \sqrt{n} \implies \Big|\EXP(e^{it\frac{S_n-nA}{\sqrt{n}}})\Big| \leq Ce^{-cD \log n} = \frac{C}{n^{cD}}$$
and
\begin{align*}
\bigg|\int_{\sqrt{\frac{D \log n}{n}} < |t|<\delta } \widehat{f}(t)\EXP(e^{it(S_n-nA)}) \, dt \bigg|  &=\bigg|\int_{\sqrt{D \log n} < |t|<\delta \sqrt{n} } \widehat{f}\left(\frac{t}{\sqrt{n}}\right)\EXP(e^{it\frac{S_n-nA}{\sqrt{n}}}) \frac{dt}{\sqrt{n}} \bigg| \\ &\leq \frac{C}{n^{cD}} \int_{\sqrt{\frac{D \log n}{n}} < |t|<\delta} |\widehat{f}(t)| \, dt = \frac{2\delta  C \|f\|_1}{n^{cD}}.
\end{align*}
Combining this with \eqref{t>delta} and choosing $D$ such that, $cD>(r+1)/2$ we have that,
\begin{equation}\label{t>logn}
\bigg| \int_{|t|>\sqrt{\frac{D\log n}{n}}} \wh{f}(t)\EXP(e^{it(S_n-nA)}) \, dt \bigg| = o(n^{-(r+1)/2}).
\end{equation}

Next, suppose $|t|<\sqrt{\frac{D \log n}{n}}$. Then,
$$\widehat{f}(t) = \sum_{j=0}^r \frac{\widehat{f}^{(j)}(0)}{j!}t^j+ \frac{t^{r+1}}{(r+1)!}\widehat{f}^{(r+1)}(\epsilon(t))$$
where $0\leq| \epsilon(t)| \leq |t|$. Note that, $$|\widehat{f}^{(r+1)}(\epsilon(t))|=\bigg|\int x^{r+1}e^{-i\epsilon(t)x}f(x) \, dx\bigg| \leq \int |x^{r+1}f(x)| \, dx \leq C_{r+1}(f).$$ Therefore,
\begin{align*}
\int_{|t|<\sqrt{D \log n}} &\widehat{f}\Big(\frac{t}{\sqrt{n}}\Big)\EXP(e^{it\frac{S_n-nA}{\sqrt{n}}}) \, dt \\ &= \sum_{j=0}^r \frac{\widehat{f}^{(j)}(0)}{j!n^{j/2}} \int_{|t|<\sqrt{D\log n}}t^j\EXP(e^{it\frac{S_n-nA}{\sqrt{n}}}) \, dt \nonumber  \\ &\phantom{aaaaaaaaa}+ \frac{1}{n^{(r+1)/2}}\frac{1}{(r+1)!}\int_{|t|<\sqrt{D\log n}}\EXP(e^{it\frac{S_n-nA}{\sqrt{n}}}) t^{r+1}  \widehat{f}^{(r+1)}\Big(\epsilon\Big(\frac{t}{\sqrt{n}}\Big)\Big)\, dt \nonumber
\end{align*}
where
\begin{align*}
\bigg|\int_{|t|<\sqrt{D\log n}} \EXP(e^{it\frac{S_n-nA}{\sqrt{n}}}) t^{r+1} &\widehat{f}^{(r+1)}\Big(\epsilon\Big(\frac{t}{\sqrt{n}}\Big)\Big)\, dt \bigg|\leq C_{r+1}(f) \int |t|^{r+1} e^{-ct^2} \, dt
\end{align*}
for large $n$. Hence,
\begin{multline}\label{TaylorfHat}
\int_{|t|<\sqrt{D \log n}} \widehat{f}\Big(\frac{t}{\sqrt{n}}\Big) \EXP(e^{it\frac{S_n-nA}{\sqrt{n}}}) \, dt \\ = \sum_{j=0}^r \frac{\widehat{f}^{(j)}(0)}{j!n^{j/2}} \int_{|t|<\sqrt{D\log n}}t^j\EXP(e^{it\frac{S_n-nA}{\sqrt{n}}}) \, dt  + C_{r+1}(f)\cO(n^{-(r+1)/2}).
\end{multline}
Because $s=r+2$, from $\eqref{PolyComp}$, 
\begin{align}\label{FreqAsympPoly}
e^{\frac{\sigma^2t^2}{2}}\EXP(e^{it \frac{S_n-nA}{\sqrt{n}}})&=\exp \Big( n\psi\Big(\frac{t}{\sqrt{n}}\Big) \Big)Z\Big(\frac{t}{\sqrt{n}}\Big) + e^{-\frac{inAt}{\sqrt{n}}+\frac{\sigma^2t^2}{2}}\ell\big(\Lambda^n_{t/\sqrt{n}}v \big) \nonumber \\  
 &= \sum_{k=0}^{r} \frac{A_k(t)}{n^{k/2}}+\frac{t^{r}}{n^{r/2}}\varphi\Big(\frac{t}{\sqrt{n}}\Big)+\cO\Big(\frac{\log^{(r+1)/2}(n)}{n^{(r+1)/2}}\Big).
\end{align}
Substituting this in \eqref{TaylorfHat},
\begin{align}\label{Near0WeakEdge}
&\int_{|t|<\sqrt{D \log n}} \widehat{f}\left(\frac{t}{\sqrt{n}}\right) \EXP(e^{it\frac{S_n-nA}{\sqrt{n}}}) \, dt \\ &= \sum_{j=0}^r \frac{\widehat{f}^{(j)}(0)}{j!n^{j/2}} \int_{|t|<\sqrt{D\log n}}t^j e^{-\sigma^2t^2/2}\sum_{k=0}^{r} \frac{A_k(t)}{n^{k/2}} \, dt + \cO\Big(\frac{\log^{(r+1)/2}(n)}{n^{(r+1)/2}}\Big) \nonumber \\ &= \sum_{k=0}^{r}  \sum_{j=0}^r  \frac{\widehat{f}^{(j)}(0)}{j!n^{(k+j)/2}} \int_{|t|<\sqrt{D\log n}}t^j A_k(t) e^{-\sigma^2t^2/2} \, dt + o(n^{-r/2}). \nonumber
\end{align}
Recall from \eqref{Parity} that $A_k$ and $k$ have the same parity. Therefore, if $k+j$ is odd then $$\int_{|t|<\sqrt{D\log n}}t^j A_k(t) e^{-\sigma^2t^2/2} \, dt = 0.$$
So only the integral powers of $n^{-1}$ will remain in the expansion. Also, there is $C$ that depends only on $r$ such that,
\begin{align*}
\int_{|t|\geq \sqrt{D\log n}}t^j A_k(t) e^{-\sigma^2t^2/2} \, dt &\leq C \int_{|t|\geq \sqrt{D\log n}}t^{4r} e^{-\sigma^2t^2/2} \, dt \leq \frac{ C}{e^{\sigma^2D\log(n)/4}} =\frac{C}{n^{\sigma^2D/4}}.
\end{align*}
Choosing $D$ such that $2\sigma^2D>(r+1)/2$, $$\int_{\reals} t^j A_k(t) e^{-\sigma^2t^2/2} \, dt= \int_{|t|\leq \sqrt{D\log n}}t^j A_k(t) e^{-\sigma^2t^2/2} \, dt + o(n^{-r/2}).$$
Therefore, fixing $D$ large, we can assume the integrals to be over the whole real line. Now, define $$a_{k,j}=\int_\reals t^j A_k(t) e^{-\sigma^2 t^2/2} \, dt $$ 
and substitute
$$\widehat{f}^{(j)}(0)= \int_{\reals} (-it)^j f(t) \, dt$$
in \eqref{Near0WeakEdge} to obtain,
\begin{align}\label{Near0Exp}
\int_{|t|<\sqrt{D \log n}}\widehat{f}\Big(\frac{t}{\sqrt{n}}\Big)\EXP(e^{it\frac{S_n-nA}{\sqrt{n}}}) \, dt &= \sum_{k=0}^{r}  \sum_{j=0}^r  a_{k,j}\frac{1}{j!n^{(k+j)/2}} \int_{\reals} (-it)^j f(t) \, dt  + o(n^{-r/2}) \\ &= \sum_{p=0}^r \frac{1}{n^p} \int_\reals f(t) \sum_{k+j=2p} \frac{a_{k,j}}{j!} (-it)^j  \, dt + o(n^{-r/2}) \nonumber \\&= \sum_{p=0}^{\lfloor r/2 \rfloor} \frac{1}{n^p} \int_\reals f(t) P_{p,l}(t) \, dt + o(n^{-r/2}) \nonumber
\end{align}
where 
\begin{equation}\label{WeakLocalPoly}
P_{p,l}(t)= \sum_{k+j=2p} \frac{a_{k,j}}{j!}(-it)^j.
\end{equation} 
The final simplification was done by absorbing the terms corresponding to higher powers of $n^{-1}$ into the error term. Note that $P_{p,l}$ is a polynomial of degree at most $2p$ and that once we know $A_0,\dots, A_{2p}$ we can compute $P_{p,l}$.  

Finally combining \eqref{Near0Exp} and \eqref{t>logn} substituting in \eqref{Plancherel} we obtain the required result as shown below.
\begin{align*}
\sqrt{n}\EXP(f(S_n-nA))&=\frac{1}{2\pi}\int_{|t|<\sqrt{D\log n}} \wh{f}\Big(\frac{t}{\sqrt{n}}\Big)\EXP(e^{it\frac{S_n-nA}{\sqrt{n}}})\, dt \\ &\phantom{aaaaaaaaaaaaaaaaaaa} + \frac{\sqrt{n}}{2\pi}\int_{|t|>\sqrt{\frac{D\log n}{n}}}\wh{f}(t)\EXP(e^{it(S_n-nA)})\, dt \\ &= \frac{1}{2\pi}\sum_{p=0}^{\lfloor r/2 \rfloor} \frac{1}{n^p} \int_\reals f(t) P_{p,l}(t) \, dt + o(n^{-r/2}) + \sqrt{n}\ o(n^{-(r+1)/2})
\\ &= \frac{1}{2\pi}\sum_{p=0}^{\lfloor r/2 \rfloor} \frac{1}{n^p} \int_\reals f(t) P_{p,l}(t) \, dt + o(n^{-r/2}).
\end{align*}
\end{proof}

The proof of \Cref{WGEdgeExp} uses the relation \eqref{FreqAsympPoly} derived in the previous proof. But we do not use the Taylor expansion of $\wh{f}$, so differentiability of $\wh{f}$ is not required. So the assumption on the decay of $f$ at infinity can be relaxed.  

\begin{proof}[Proof of Theorem \ref{WGEdgeExp}]
Multiplying \eqref{FreqAsympPoly} by $\widehat{f}$ and integrating we obtain,
\begin{multline*}
\int_{|t|<\sqrt{D\log n}}\widehat{f}\Big(\frac{t}{\sqrt{n}}\Big)\EXP(e^{it \frac{S_n-nA}{\sqrt{n}}})\, dt \\ = \sum_{k=0}^{r}\frac{1}{n^{k/2}} \int_{|t|<\sqrt{D\log n}}\widehat{f}\Big(\frac{t}{\sqrt{n}}\Big) A_k(t) e^{-\frac{\sigma^2t^2}{2}}  \, dt + \|f\|_1o(n^{-r/2}).
\end{multline*}

As in the proof of Theorem \ref{WLEdgeExp} the integrals above can be replaced by integrals over $\reals$ without altering the order of the error because 
$$ \int_{|t|\geq\sqrt{D\log n}}\widehat{f}\Big(\frac{t}{\sqrt{n}}\Big) A_k(t) e^{-\frac{\sigma^2t^2}{2}}  \, dt \leq \|f\|_1 \, o(n^{-r/2})$$
for $D$ such that $2\sigma^2D>(r+1)/2$. Therefore,
\begin{align*}
\int_{|t|<\sqrt{D\log n}}\widehat{f}\Big(\frac{t}{\sqrt{n}}\Big)\EXP(e^{it \frac{S_n-nA}{\sqrt{n}}})\, dt = \sum_{k=0}^{r}\frac{1}{n^{k/2}} \int_{\reals}\widehat{f}\Big(\frac{t}{\sqrt{n}}\Big) A_k(t)e^{-\frac{\sigma^2t^2}{2}}  \, dt + \|f\|_1o(n^{-r/2}).
\end{align*}
We pick $R_p$ as in \eqref{PolyForDensity} and claim $P_{p,g}=R_p$.

Note that $\sqrt{n} f(t\sqrt{n}) \longleftrightarrow  \widehat{f}(t/\sqrt{n}) $. So by the Plancherel theorem,
$$\int_\reals \sqrt{n}f\left(t\sqrt{n}\right)R_k(t)\fn(t) \, dt = \frac{1}{2\pi}\int_{\reals}\widehat{f}\Big(\frac{t}{\sqrt{n}}\Big) A_k(t) e^{-\frac{\sigma^2t^2}{2}} \, dt .$$
Thus, 
\begin{align}\label{GlobalNear0Asymp}
\frac{1}{2\pi\sqrt{n}}\int_{|t|<\sqrt{D\log n}}&\widehat{f}\Big(\frac{t}{\sqrt{n}}\Big) \EXP(e^{it\frac{S_n-nA}{\sqrt{n}}}) \, dt \nonumber \\ &= \frac{1}{\sqrt{n}}\Big(\sum_{p=0}^{r}\frac{1}{n^{p/2}} \int_{\reals} \sqrt{n}f\big(t\sqrt{n}\big)R_p(t)\fn(t)\, dt+\|f\|_1 o(n^{-r/2})\Big) \nonumber \\ &=\sum_{p=0}^{r}\frac{1}{n^{p/2}} \int_{\reals} f\big(t\sqrt{n}\big)R_p(t)\fn(t)\, dt + \|f\|_1 o(n^{-(r+1)/2}).
\end{align}
Note that \eqref{t>logn} holds because $f \in F_{0}^{q+2}$. Now, combining \eqref{GlobalNear0Asymp} with the estimate \eqref{t>logn} completes the proof.
\end{proof}

\begin{rem}\label{AltPrfWEdge}
Proofs of both the \Cref{NewWLEdgeExp} and \Cref{NewWGEdgeExp} are almost identical except the estimate  \eqref{AtInfty}. In order to obtain the same asymptotics, the assumption on the integrability of $\wh{f^{(q)}}$ can be replaced by $(A5)$ and the fact that $|\widehat{f}(t)| \sim \frac{1}{t}$ as $t \to \pm \infty$.
\begin{align*}
\bigg| \int_{|t|>n^{r_1}} \wh{f}(t)\EXP(e^{it(S_n-nA)}) \, dt \bigg| &\leq C\int_{|t|>n^{r_1}} |\wh{f}(t)| \|\cL^n_t\|\, dt \\ &\leq C\|f\|_1\int_{|t|>n^{r_1}}  \frac{1}{t^{1+\alpha}} \, dt  \\ &\leq  \frac{C\|f\|_1}{n^{r_1(\alpha-\epsilon)}}\int \frac{1}{t^{1+\epsilon}}\, dt 
\end{align*}
Since, $r_1\alpha > \frac{r+1}{2}$ choosing $\epsilon$ small enough we can make the expression $\|f\|_1\hspace{2pt}o(n^{-(r+1)/2})$ as required. 
\end{rem}

\begin{proof}[Proof of Theorem \ref{AVGEdgeExp}]
Select $A$ as in \eqref{CLT}. Define $P_p$ by \eqref{PolyForDensity} and \eqref{EdgePolyRel} 
and $\tilde{f}_n(x)=f(-\sqrt{n}x)$. Then the change of variables $-\frac{y}{\sqrt{n}} \to y$ yields, 
\begin{align*}
\int \Big[\Prob\Big(\frac{S_n-nA}{\sqrt{n}}\leq x+\frac{y}{\sqrt{n}}\Big)-\fN\Big(x+\frac{y}{\sqrt{n}}\Big)-\cE_{r,n}\Big(&x+\frac{y}{\sqrt{n}}\Big)\Big] f(y) dy \\ &= \sqrt{n}\Delta_n \ast \tilde{f}_n (x).
\end{align*}
where $\cE_{r,n}(x)=\sum_{p=1}^r \frac{1}{n^{p/2}} P_p(x)\fn(x)$. 

Notice that $\EXP(e^{it\frac{S_n-nA}{\sqrt{n}}})\wh{\tilde{f}_n} \in L^1$. Therefore, $$(F_n \ast \tilde{f}_n)'(x)=\frac{1}{2\pi}\int e^{-itx}\EXP(e^{it\frac{S_n-nA}{\sqrt{n}}})\wh{\tilde{f}_n}(t) \, dt. $$
Also,
$$\Big[\fn+ \Big(\sum_{p=1}^r \frac{1}{n^{p/2}} R_p\fn\Big)\Big]\ast \tilde{f}_n(x) =  \frac{1}{2\pi}\int e^{-itx} e^{-\frac{\sigma^2t^2}{2}}\big(1+Q_n(t)\big)\wh{\tilde{f}_n}(t) \, dt$$ where $R_p$'s are polynomials given by \eqref{PolyForDensity} and $Q_n(t)$ is given by \eqref{MainPolyExp}. From these we conclude that, 
\begin{align}\label{FourierInvDensity}
(\Delta_n \ast \tilde{f}_n)' (x) = \frac{1}{2\pi}\int e^{-itx}\big(\EXP(e^{it\frac{S_n-nA}{\sqrt{n}}})- e^{-\frac{\sigma^2t^2}{2}}\big(1+Q_n(t)\big)\wh{\tilde{f}_n}(t) \, dt.
 \end{align}
We claim that,
\begin{align}\label{FourierInvDistr}
(\Delta_n \ast \tilde{f}_n) (x) &= \frac{1}{2\pi} \int e^{-itx}\frac{\EXP(e^{it\frac{S_n-nA}{\sqrt{n}}})- e^{-\frac{\sigma^2t^2}{2}}\big(1+Q_n(t)\big)}{-it}\wh{\tilde{f}_n}(t)   \, dt.  
\end{align}

Indeed, if the right side of \eqref{FourierInvDistr} converges absolutely, then Riemann-Lebesgue Lemma gives us that it converges $0$ as $|x| \to \infty$. Differentiating \eqref{FourierInvDistr} we obtain \eqref{FourierInvDensity}. Thus the two sides in \eqref{FourierInvDistr} can differ only by a constant. Since both are $0$ at $\pm\infty$, this constant is $0$ and \eqref{FourierInvDistr} holds.


Now, we are left with the task of showing that the right side of \eqref{FourierInvDistr} converges absolutely. From the definition of $\tilde{f}_n$ it follows that, $\wh{\tilde{f}_n}(t) = \frac{1}{\sqrt{n}} \wh{f}\big(-\frac{t}{\sqrt{n}}\big).$ Combining this with \eqref{FullNear0Est}, we have that,
\begin{align*}
\bigg|\int_{|t|<\delta \sqrt{n}} e^{-itx} &\frac{\EXP(e^{it\frac{S_n-nA}{\sqrt{n}}})- e^{-\frac{\sigma^2t^2}{2}}\big(1+Q_n(t)\big)}{-it}\wh{\tilde{f}_n}(t)   \, dt \bigg| \\ &\leq \int_{|t|<\delta \sqrt{n}}\bigg|\frac{\EXP(e^{it\frac{S_n-nA}{\sqrt{n}}})- e^{-\frac{\sigma^2t^2}{2}}\big(1+Q_n(t)\big)}{t}\wh{\tilde{f}_n}(t)  \bigg|   \, dt \\ & \leq \frac{\|f\|_1}{\sqrt{n}} \int_{|t|<\delta \sqrt{n}}\bigg|\frac{\EXP(e^{it\frac{S_n-nA}{\sqrt{n}}})- e^{-\frac{\sigma^2t^2}{2}}\big(1+Q_n(t)\big)}{t}  \bigg|   \, dt \\ &=  \|f\|_1 o(n^{-(r+1)/2}).
\end{align*} 
Note that, 
\begin{align}
\bigg|\int_{|t|>\delta \sqrt{n}}e^{-itx} &\frac{\EXP(e^{it\frac{S_n-nA}{\sqrt{n}}})- e^{-\frac{\sigma^2t^2}{2}}\big(1+Q_n(t)\big)}{-it}\wh{\tilde{f}_n}(t)   \, dt \bigg| \nonumber \\ 
&\leq \int_{|t|>\delta \sqrt{n}}\bigg|\frac{\EXP(e^{it\frac{S_n-nA}{\sqrt{n}}})- e^{-\frac{\sigma^2t^2}{2}}\big(1+Q_n(t)\big)}{t} \wh{f}\Big(-\frac{t}{\sqrt{n}}\Big) \bigg|  \, dt \nonumber \\ 
&\leq \frac{1}{\sqrt{n}} \int_{|t|>\delta} \bigg|\frac{\EXP(e^{-it(S_n-nA)})- e^{-\frac{n^2\sigma^2t^2}{2}}\big(1+Q_n(-\sqrt{n}t)\big)}{t} \wh{f}(t)\bigg|   \, dt \nonumber \\ 
&\leq \frac{1}{\sqrt{n}} \int_{|t|>\delta} \bigg|\frac{\EXP(e^{-it(S_n-nA)})} {t}\wh{f}(t)\bigg|   \, dt + \cO(e^{-cn^2}). \nonumber
\end{align} 
Put, 
\begin{equation*}
J_n=\frac{1}{\sqrt{n}}\int_{|t|>\delta} \bigg|\frac{\EXP(e^{-it(S_n-nA)})} {t}\wh{f}(t)\bigg|   \, dt.
\end{equation*}
We claim $J_n=o(n^{-(r+1)/2})$. This proves that \eqref{FourierInvDistr} converges absolutely as required.

To conclude the asymptotics of $J_n$, choose $\overline{\delta}> \max \{\delta,K\}$ where $K$ as in (A4). From (A3) there exists $\gamma<1$ such that $\|\cL_{t}^n\|\leq \gamma^n $ for all $\delta \leq |t| \leq \overline{\delta}$ for sufficiently large $n$. Then, using \eqref{MainAssum} for sufficiently large $n$ we have, 
\begin{align*} 
\frac{1}{\sqrt{n}}\int_{\delta <|t|<\overline{\delta}}\left|\frac{\EXP(e^{-it(S_n-nA)})}{t}\wh{f}(t)\right|\, dt \leq \frac{C\|f\|_1}{\delta \sqrt{n}}\int_{\delta <|t|<\overline{\delta}}\|\cL^n_{t}\| \, dt = \cO(\gamma^n). 
\end{align*}
Next, for $K\leq\overline{\delta}\leq |t| \leq n^{r_1}$, $\|\cL_{t}^n\|\leq \frac{1}{n^{r_2}} $. 
Hence, for $n$ sufficiently large so that $r_2>\frac{r}{2}$,
\begin{align*}
\frac{1}{\sqrt{n}}\int_{\overline{\delta}<|t|<n^{r_1}}\left|\frac{\EXP(e^{-it(S_n-nA)})}{t}\wh{f}(t)\right|\, dt &\leq \frac{C}{\delta \sqrt{n}}\int_{\overline{\delta}<|t|<n^{r_1}}\|\cL^n_{t}\||\wh{f}(t)| \, dt \\ &\leq \frac{C\|\wh{f}\|_1}{n^{r_2+1/2}}=o(n^{-(r+1)/2})
\end{align*} 
Since $q>\frac{r}{2r_1}$, we have that,
\begin{align*}
\frac{1}{\sqrt{n}}\int_{|t|>n^{r_1}}\left|\frac{\EXP(e^{-it(S_n-nA)})}{t}\wh{f}(t)\right|\, dt &\leq \frac{\|f^{(q)}\|_1}{\sqrt{n}}\int_{|t|>n^{r_1}}\frac{1}{|t|^{q+1}} \, dt \leq \frac{C\|f^{(q)}\|_1}{n^{qr_1+1/2}}=o(n^{-(r+1)/2})
\end{align*} 
Combining the above estimates, $J_n=C^q(f)o(n^{-(r+1)/2})$.

This completes the proof that $(\Delta_n \ast \tilde{f}_n) (x)=o(n^{-(r+1)/2})$. Hence, 
\begin{align*}
\int \Big[\Prob\Big(\frac{S_n-nA}{\sqrt{n}}&\leq x+\frac{y}{\sqrt{n}}\Big)-\fN\Big(x+\frac{y}{\sqrt{n}}\Big)\Big)\Big] f(y) dy \\ &= \int\cE_{r,n}\Big(x+\frac{y}{\sqrt{n}}\Big) f(y)\, dy + \sqrt{n}\Delta_n \ast \tilde{f}_n (x) \\ &= \sum_{p=1}^r \frac{1}{n^{p/2}} \int P_p\Big(x+\frac{y}{\sqrt{n}}\Big)\fn(x) f(y)\, dy + C^{q}(f)o(n^{-r/2})
\end{align*}
as required.
\end{proof}

In the lattice case, periodicity allows us to simplify the proof significantly although the idea behind the proof is similar to previous proofs. 

\begin{proof}[Proof of Theorem \ref{LatticeEdgeExp}]
Under assumptions (A1) and (A2) we have the CLT for $S_n$. Put $A$ as in \eqref{CLT}. We observe that, 
\begin{align*}
2\pi\Prob(S_n=k) &= \int_{-\pi}^{\pi} e^{-itk} \EXP(e^{itS_n}) \, dt =\int_{-\pi}^{\pi} e^{-itk} \ell(\cL^n_t v) \, dt. \nonumber 
\end{align*}
After changing variables and using \eqref{eq:char fn}, \eqref{av proj} we have,
\begin{align}\label{LatticeFT}
2\pi \sqrt{n}\Prob\left(S_n=k\right)= \int_{-\pi \sqrt{n}}^{\pi\sqrt{n}} e^{-\frac{itk}{\sqrt{n}}} \mu\Big(\frac{t}{\sqrt{n}}\Big)^nZ\Big( \frac{t}{\sqrt{n}} \Big)\, dt  + \int_{-\pi \sqrt{n}}^{\pi\sqrt{n}} e^{-\frac{itk}{\sqrt{n}}} \ell\big(\Lambda^n_{t/\sqrt{n}} v\big)  \, dt. 
\end{align}
By $\widetilde{(\text{A3})}$ there exists $C>0$ and $r \in (0,1)$ (both independent of $t$) such that $|\ell\left(\Lambda^n_{t} v\right)|\leq Cr^n$ for all $t\in[-\pi,\pi]$. Therefore the second term of \eqref{LatticeFT} decays exponentially fast to $0$ as $n \to \infty$. 

Now, we focus on the first term. Using the same strategy as in the proof of Theorem \ref{EdgeExp} we have,
\begin{align}\label{LatticeEigenExp}
\mu\Big(\frac{t}{\sqrt{n}}\Big)^nZ\Big( \frac{t}{\sqrt{n}}\Big) &=  e^{\frac{inAt}{\sqrt{n}}-\frac{\sigma^2t^2}{2}} \left[1+ Q_n(t)+ o(n^{-r/2})\right]
\end{align}
where $Q_n(t)$ is as in \eqref{MainPolyExp}. Define $R_j$ as in \eqref{PolyForDensity}.
\begin{align*}
2\pi \sqrt{n} &\Prob(S_n=k) - 2\pi \bigg\{\frac{1}{\sqrt{2\pi}}e^{-\frac{(k-nA)^2}{2\sigma^2n}}\bigg(1+\sum_{j=1}^{r}\frac{(R_p(k-nA)/\sqrt{n})}{n^{j/2}}\bigg)\bigg\} \\ 
&= \int_{-\pi \sqrt{n}}^{\pi\sqrt{n}} e^{-\frac{itk}{\sqrt{n}}} \mu\Big(\frac{t}{\sqrt{n}}\Big)^nZ\Big( \frac{t}{\sqrt{n}} \Big)\, dt \\ &\phantom{aaaaa}- \int_{-\infty}^\infty e^{-\frac{it(k-nA)}{\sqrt{n}}}e^{-\sigma^2t^2/2}\, dt - \int_{-\infty}^\infty e^{-\frac{itk}{\sqrt{n}}} e^{-\frac{\sigma^2t^2}{2}}Q_n(t)\, dt +o(n^{-r/2}).
\end{align*}
We estimate the RHS by estimating the three integrals given below, 
\begin{align*}
I_1 &=\int_{-\delta \sqrt{n}}^{\delta \sqrt{n}} e^{-\frac{itk}{\sqrt{n}}} \mu\Big(\frac{t}{\sqrt{n}}\Big)^nZ\Big( \frac{t}{\sqrt{n}} \Big) - e^{-\frac{it(k-nA)}{\sqrt{n}}} e^{-\frac{\sigma^2t^2}{2}}[1+Q_n(t)]  \, dt \\ I_2&=\int_{\delta\sqrt{n}<|t|<\pi \sqrt{n}} e^{-\frac{itk}{\sqrt{n}}} \mu\Big(\frac{t}{\sqrt{n}}\Big)^nZ\Big( \frac{t}{\sqrt{n}} \Big) \, dt \\ I_3&=\int_{|t|>\delta\sqrt{n}} e^{-\frac{it(k-nA)}{\sqrt{n}}}e^{-\frac{\sigma^2t^2}{2}}[1+Q_n(t)]  \, dt.
\end{align*}
Clearly, $|I_3|$ decays to $0$ exponentially fast as $n \to \infty$. Also, $|\mu(2\pi)|=1$ and $|\mu(t)|\in (0,1)$ for $0<|t|<2\pi$. Therefore, there exists $\epsilon>0$ such that $|\mu(t)| < \epsilon$ on $\delta \leq |t| \leq \pi$. Put $M= \max_{\delta\leq |t| \leq \pi}{|Z(t)|}$. Then, $$|I_2| \leq M \sqrt{n} \int_{\epsilon < |t| < \pi} |\mu(t)|^n \, dt \leq 2M(\pi - \delta) \sqrt{n} \epsilon^n. $$ Hence, $|I_2|$ decays to $0$ exponentially fast as $n \to \infty$. From \eqref{LatticeEigenExp}, we have that
\begin{align*}
 e^{-\frac{itk}{\sqrt{n}}} \Big[ \mu\Big(\frac{t}{\sqrt{n}}\Big)^nZ\Big( \frac{t}{\sqrt{n}}\Big) - e^{\frac{inAt}{\sqrt{n}}}e^{-\frac{\sigma^2t^2}{2}}[1+Q_n(t)]\Big] = e^{-\frac{\sigma^2t^2}{2}} o(n^{-r/2}).
\end{align*}
This implies $|I_1|=o(n^{-r/2})$. Combining these estimates we have the required result.
\end{proof}

\section{Computing coefficients}\label{Coeff}
Since $\int_{|t|>\delta} \EXP(e^{itS_n}) \, dt$ decays sufficiently fast, the Edgeworth expansion, and hence its coefficients, depend only on the Taylor expansion of $\EXP(e^{itS_n})$ about $0$. Here we relate the coefficients of Edgeworth polynomials to the asymptotics of moments of $S_n$ by relating them to derivatives of $\mu(t)$ and $Z(t)$ at $0$. 

Suppose (A1) through (A4) are satisfied with $s=r+2$. Recall \eqref{eq:char fn}:
\begin{equation}\label{CharFn}
\EXP(e^{it S_n})=\mu\left(t\right)^n \ell\left( \Pi_{t} v \right) + \ell\left(\Lambda^n_{t} v\right).
\end{equation}
Put $Z(t)=\ell\left( \Pi_{t} v \right)$ as before. Also write $U_n(t)=\ell\left(\Lambda^n_{t} v\right)$. We already know that $\mu(t),Z(t)$ and $U(t)$ are $r+2$ times continuously differentiable. Using \eqref{ResidualProj} one can show further that the derivatives of $U_n(t)$ satisfy: $$\sup_{|t|\leq \delta}\|U^{(k)}_n\|\leq C\ve_0^n$$ for all $n$ and for all $1 \leq k \leq r+2$. 

Taking the first derivative of \eqref{CharFn} at $t=0$ we have: 
\begin{align*}
i\EXP(S_n)&=n\mu'(0)+Z'(0)+U_n'(0)  \implies \lim_{n\to \infty} i \EXP\Big( \frac{S_n}{n} \Big) =\mu'(0).
\end{align*}
In fact, using the Taylor expansion of $\log \mu(t)$ and above limit one can conclude that the number $A$ we used in the statement of the CLT in \eqref{CLT}, is given by $$A=\lim_{n\to \infty} \EXP\Big( \frac{S_n}{n} \Big).$$ Therefore one can rewrite $\eqref{eq:char fn}$ as 
\begin{equation}\label{eq:char fn new}
\EXP(e^{it (S_n-nA)})=e^{-nt\mu'(0)}\mu\left(t\right)^n Z(t) + \overline{U}_n(t)
\end{equation}
where $\overline{U}_n(t)=e^{-nt\mu'(0)}U_n(t)$. Also note that its derivatives satisfy $\|\overline{U}^{(k)}_n\|_\infty=\cO(\ve^n_0)$ for all $1 \leq k \leq r+2$.

From \eqref{eq:char fn new}, it follows that moments of $S_n-nA$ can be expanded in powers of $n$ with coefficients depending on derivatives of $\mu$ and $Z$ at $0$. However, only powers of $n$ upto order $k/2$ will appear. We prove this fact below. 

\begin{lem}\label{MomentExp}Let $1\leq k \leq r+2$. Then for large $n$,
\begin{equation}\label{MomentExp1}
\EXP\big(\left[S_n-nA\right]^k\big)=\sum_{j=0}^{\lfloor k/2 \rfloor}a_{k,j}n^{j}+\cO(\epsilon^n_0).
\end{equation}
\end{lem}

\begin{proof}
We first note that taking the $k$th derivative of \eqref{eq:char fn new} at $t=0$, 
\begin{align*}
i^k\EXP\big(\left[S_n-nA\right]^k\big) &=\frac{d^k}{dt^k}\bigg|_{t=0}\left[e^{-nt\mu'(0)}\mu\left(t\right)^n Z(t)\right] + \overline{U}^{(k)}(0) \\ &= \frac{d^k}{dt^k}\bigg|_{t=0}\left[e^{-nt\mu'(0)}\mu\left(t\right)^n  Z(t)\right] + \cO(\epsilon^n_0).
\end{align*}
Observe that all the derivatives of $e^{-nt\mu'(0)}\mu\left(t\right)^n Z(t)$ will only have positive integral powers of $n$ (possibly) up to order $k$. Therefore, $\frac{d^k}{dt^k}\big|_{t=0}\left[e^{-nt\mu'(0)}\mu\left(t\right)^n  Z(t)\right]=\sum_{j=0}^{k}a_{k,j}n^{j}$. We claim that for $j>k/2$, $a_{k,j}=0$. This claim proves the result. 

We notice that the first derivative of $e^{-t\mu'(0)}\mu\left(t\right)$ at $t=0$ is $0$. Thus we prove the more general claim that if $g(0)=1$ and $g'(0)=0$ then $\frac{d^k}{dt^k}\big|_{t=0}[g(t)^nZ(t)]$ has no terms with powers of $n$ greater than $k/2$. From the Leibniz rule,
\begin{align*}
\frac{d^k}{dt^k}\bigg|_{t=0}[g(t)^nZ(t)]&=\sum_{l=0}^k {k \choose l} Z^{(k-l)}(0)\frac{d^l}{dt^l}\bigg|_{t=0}[g(t)^n].
\end{align*}
Therefore it is enough to prove that $\frac{d^l}{dt^l}\big|_{t=0}[g(t)^n]$ has no powers of $n$ greater than $l/2$. 

To this end we use the order $l$ Taylor expansion of $g(t)$ about $t=0$. Since $g'(0)=0$ and $g$ is $r+2$ times continuously differentiable for $l\leq r+2$ there exists $\phi(t)$ continuous such that,
\begin{align*}
&\phantom{aaaaaa\hspace{2pt}} g(t)= 1+a_2t^2+\dots+a_lt^l+t^{l+1}\phi(t) \\
&\implies g(t)^n = \sum_{k_0+k_2+\dots+k_{l+1}=n} \frac{n!}{k_0!k_2!\dots k_{l+1}!} (a_2t^2)^{k_2} \dots t^{(l+1)k_{l+1}}\phi(t)^{k_{l+1}} \\
&\phantom{\implies g(t)^n\hspace{3pt}} = \sum_{k_0+k_2+\dots+k_{l+1}=n}\frac{ C_{k_0k_2\dots k_{l+1}} n!}{k_0!k_2!\dots k_{l+1}!} t^{2k_2+\dots+(l+1)k_{l+1}}\phi(t)^{k_{l+1}}.
\end{align*}

After combining and rearranging terms according to powers of $t$, we can obtain the order $l$ Taylor expansion of $g(t)^n$. Notice that if $k_{l+1}\geq 1$ then $2k_2+\dots+(l+1)k_{l+1} \geq l+1$. Terms with $k_{l+1}\geq 1$ are part of the error term of the order $l$ Taylor expansion of $g(t)^n$. Since our focus is on the derivative at $t=0$, the only terms that matter are terms with $k_{l+1}=0$ and $2k_2+\dots+lk_{l}=l$. This implies that $k_2+\dots+k_l\leq \frac{l}{2}$. Because $k_i$'s are non-negative integers, this means $k_2+\dots+k_l\leq \lfloor \frac{l}{2} \rfloor$. Hence, $k_0 \geq n-\lfloor \frac{l}{2} \rfloor$. 

This analysis shows that the largest contribution to $\frac{d^l}{dt^l}\big|_{t=0}[g(t)^n]$ comes from the term,
$$\frac{ C_{(n-\lfloor \frac{l}{2} \rfloor),1,\dots,1,0,\dots,0}\ n!}{\big(n-\lfloor \frac{l}{2} \rfloor \big)!}\ t^{l}$$
whose $k$th derivative at $0$ is,
\begin{align*}
\frac{  C_{(n-\lfloor \frac{l}{2} \rfloor),1,\dots,1,0,\dots,0}\ l!\ n!}{\big(n-\lfloor \frac{l}{2} \rfloor \big)!} &= C_{(n-\lfloor \frac{l}{2} \rfloor),1,\dots,1,0,\dots,0}\ l! \ n\dots \Big(n-\Big\lfloor \frac{l}{2} \Big\rfloor +1 \Big) = \cO(n^{\lfloor \frac{l}{2} \rfloor}).
\end{align*}
Therefore, 
$$\frac{d^l}{dt^l}\Big|_{t=0}[g(t)^n]=\cO(n^{\lfloor \frac{l}{2} \rfloor}).$$ 
\end{proof}
It is immediate from the proof that the coefficients $a_{k,j}$ are determined by the derivatives of $\mu(t)$ and $Z(t)$ near $0$. For example, the constant term $a_{k,0}=(-i)^k Z^{(k)}(0)$. This follows from the following three facts. The expansion \eqref{MomentExp1} is the $k$th derivative of the product of the three functions $e^{-nt\mu'(0)}, \mu\left(t\right)^n$ and $Z(t)$ at $t=0$. All derivatives of $\mu\left(t\right)^n$ and $e^{-nt\mu'(0)}$ at $t=0$ contain powers of $n$ and thus, $a_{k,0}$ corresponds to the term $Z(t)$ being differentiated $k$ times in the Leibneiz rule. Both $e^{-nt\mu'(0)}$ and $\mu\left(t\right)^n$ are $1$ at $t=0$. We will see later that the other coefficients $a_{k,j}$ are combinations of $\mu'(0)=iA$, higher order derivatives of $\mu$ at $0$ upto order $k$ and derivatives of $Z$ at $0$ upto order $k-1$.

As a corollary to \Cref{MomentExp}, we conclude that asymptotic moments of orders upto $r+2$ exist. These provide us an alternative way to describe $a_{k,j}$. 
\begin{cor}\label{Moments} For all $1 \leq m \leq r+2$ and $0\leq j \leq \frac{m}{2}$, 
\begin{align*}
a_{m,j} =\lim_{n\to\infty} \frac{\EXP\left(\left[S_n-nA\right]^{m}\right)-n^{j+1}a_{m,j+1}-\dots -n^{\lfloor \frac{m}{2}\rfloor}a_{m,\lfloor \frac{m}{2}\rfloor}}{n^{j}}. 
\end{align*}
\end{cor}
\begin{proof}
When $m=1$, $\EXP([S_n-nA])=a_{1,0}+\cO(\epsilon^n_0)$ and it is immediate that $a_{1,0}=\lim_{n \to \infty}\EXP([S_n-nA])$. For arbitrary $k$ we have, 
$$
\EXP\big(\left[S_n-nA\right]^k\big)=a_{k,\lfloor k/2 \rfloor}n^{\lfloor k/2 \rfloor}+a_{k,\lfloor k/2 \rfloor-1}n^{\lfloor k/2 \rfloor-1}+\dots+a_{k,0}+\cO(\epsilon^n_0) 
$$
and dividing by $n$ we obtain,
 $$\frac{\EXP\big(\left[S_n-nA\right]^k\big)}{n^{\lfloor k/2 \rfloor}}=a_{k,\lfloor k/2 \rfloor} + \cO\Big(\frac{1}{n}\Big).$$
Now, it is immediate that,
$$a_{k,\lfloor k/2\rfloor}=\lim_{n\to\infty} \frac{\EXP\big(\left[S_n-nA\right]^{k}\big)}{n^{\lfloor  k/2 \rfloor}}.$$  
Having computed $a_{k,j}$, for $r\leq j\leq \lfloor \frac{k}{2}\rfloor$, we can write, 
\begin{align*}
\EXP\big(\left[S_n-nA\right]^k\big)-a_{k,\lfloor k/2 \rfloor}n^{\lfloor k/2 \rfloor}-\dots-a_{k,r}n^{r} =a_{k,r-1}n^{r-1}+\dots+a_{k,0}+\cO(\epsilon^n_0).
\end{align*}
Dividing by $n^{r-1}$, we obtain, 
$$\frac{\EXP\big(\left[S_n-nA\right]^{k}\big)-n^{r}a_{k,r}-\dots -n^{\lfloor k/2\rfloor}a_{k,\lfloor k/2\rfloor}}{n^{r-1}} = a_{k,r-1}+\cO\Big(\frac{1}{n}\Big).$$
Now, we can compute $a_{m+1,r-1}$,
\begin{align*}
&a_{k,r-1} =\lim_{n\to\infty} \frac{\EXP\big(\left[S_n-nA\right]^{k}\big)-n^{r}a_{k,r}-\dots -n^{\lfloor k/2\rfloor}a_{k,\lfloor k/2\rfloor}}{n^{r-1}}. 
\end{align*}
This proves the Corollary for arbitrary $k \in \{1,\dots,r+2\}$. 
\end{proof}

Because the coefficients of polynomials $A_p(t)$ (see \eqref{MainPolyExp}) are combinations of derivatives of $\mu(t)$ and $Z(t)$ at $t=0$, we can write them explicitly in terms of $a_{k,j}$, and hence, by applying \Cref{Moments}, the coefficients of Edgeworth polynomials can be expressed in terms of moments of $S_n$. Next, we will introduce a recursive algorithm to do this and illustrate the process by computing the first and second Edgeworth polynomials. 

Taking the first derivative of \eqref{eq:char fn new} at $t=0$,
$$i\EXP([S_n-nA])=Z'(0)+\overline{U}'_n(0).$$
Then, 
$$a_{1,0}=\lim_{n \to \infty}\EXP([S_n-nA])=-iZ'(0).$$
Next, taking the second derivative of \eqref{eq:char fn new} at $t=0$ we have,
$$
i^2\EXP([S_n-nA]^2)= n[\mu''(0)-\mu'(0)^2]+Z''(0)+\overline{U}''_n(0).
$$
Therefore, dividing by $n$ and taking the limit we have,
$$
a_{2,1}=\sigma^2=\lim_{n \to \infty} \EXP\left(\left[\frac{S_n-nA}{\sqrt{n}}\right]^2\right) =\mu'(0)^2 - \mu''(0).
$$
Once we have found $a_{2,1}$ we can find $$a_{2,0}=\lim_{n \to \infty}\big(\EXP([S_n-nA]^2)-n\sigma^2 \big) =-Z''(0).$$
We can repeat this procedure iteratively. For example, after we compute the $3$rd derivative of \eqref{eq:char fn new} at $t=0$:
\begin{multline*}
i^3\EXP([S_n-nA]^3)=Z^{(3)}(0)+ n\mu'(0)[2\mu'(0)^2-3\mu''(0)]+n\mu^{(3)}(0)\\ +3nZ'(0)[\mu'(0)^2-\mu''(0)]+\overline{U}^{(3)}_n(0)
\end{multline*}
we get that,
\begin{align*}
a_{3,1}=\lim_{n \to \infty} \frac{1}{n}\EXP\left(\left[S_n-nA\right]^3\right)&=-A(3\sigma^2+A^2)+i\mu^{(3)}(0)-3i\sigma^2Z'(0)\\
&=-A(3\sigma^2+A^2)+i\mu^{(3)}(0)+3\sigma^2a_{1,0}.
\end{align*}
This gives us $\mu^{(3)}(0)$ and $Z^{(3)}(0)$ in terms of asymptotics of moments of $S_n$:
$$i\mu^{(3)}(0)=a_{3,1}+A(3\sigma^2+A^2)-3\sigma^2a_{1,0}$$
$$iZ^{(3)}(0)=\lim_{n\to \infty}\big(\EXP([S_n-nA]^3)-na_{3,1}\big).$$

Given that we have all the coefficients $a_{k,j}$, $1\leq k \leq m$ computed and $\mu^{(k)}(0), Z^{(k)}(0)$ for $1\leq k \leq m$ expressed in terms of the former, we can compute $a_{m+1,j}$ and express $\mu^{(m+1)}(0), Z^{(m+1)}(0)$ in terms of $a_{k,j}$, $1\leq k \leq m+1$. 

To see this note that $\mu^{(m+1)}(0)$ appears only as a result of $\mu^{n}(t)$ being differentiated $m+1$ times. So, $\mu^{(m+1)}(0)$ only appears in derivatives of order $m+1$ and higher. It is also easy to see that it appears in the form $n\mu^{(m+1)}(0)$ in the $(m+1)$th derivative of \eqref{eq:char fn new}. Thus, it is a part of $a_{m+1,1}$ and all the other terms in $a_{m+1,1}$ are products of $\mu^{(k)}(0), Z^{(k)}(0)$ for $1\leq k \leq m$ whose orders add upto $m+1$ and hence they are products of  $a_{k,j}$, $1\leq k \leq m$. 

Also, $Z^{m+1}(0)$ appears only in $a_{m+1,0}$. This is because $Z^{m+1}(0)$ appears only as a result of $Z(t)$ being differentiated $m+1$ times. Thus, it appears only in derivatives of \eqref{eq:char fn new} of order $m+1$ or higher. In the $(m+1)$th derivative of \eqref{eq:char fn new}, there is only one term containing $Z^{(m+1)}(t)$ and it is $e^{-nt\mu'(0)}\mu\left(t\right)^n Z^{m+1}(t)$. So $a_{m+1,0}=(-i)^{m+1} Z^{m+1}(0)$.

Using \Cref{Moments}, we have,
$$a_{m+1,\lfloor \frac{m+1}{2}\rfloor}=\lim_{n\to\infty} \frac{\EXP\left(\left[S_n-nA\right]^{m+1}\right)}{n^{\lfloor \frac{m+1}{2}\rfloor}} .$$  
Having computed $a_{m+1,j}$, for $r\leq j\leq \lfloor \frac{m+1}{2}\rfloor$, we compute $a_{m+1,r-1}$:
\begin{align*}
&a_{m+1,r-1}=\lim_{n\to\infty} \frac{\EXP\left(\left[S_n-nA\right]^{m+1}\right)-n^{r}a_{m+1,r}-\dots -n^{\lfloor \frac{m+1}{2}\rfloor}a_{m+1,\lfloor \frac{m+1}{2}\rfloor}}{n^{r-1}}. 
\end{align*}

This gives us $Z^{(m+1)}(0)=i^{m+1}a_{m+1,0}$ and $\mu^{m+1}(0)$ in terms of $a_{m+1,1}$ and  $a_{k,j}$, $1\leq k \leq m$ i.e.\hspace{3pt}explicitly in terms of moments of $S_n$. Proceeding inductively we can compute all the derivatives upto order $r$ of $\mu(t)$ and $Z(t)$ at $t=0$ in this manner  by taking derivatives up to order $r$ of \eqref{eq:char fn new} at $t=0$. This is possible because our assumptions guarantee the existence of the first $r+2$ derivatives of \eqref{eq:char fn new} near $t=0$.  
\begin{rem}
This representation of $\mu^{(k)}(0)$ and $Z^{(k)}(0)$ in terms of $a_{k,j}$ is not unique. However, it is convenient to choose the $a_{k,j}$'s with the lowest possible indices. The inductive procedure explained above yields exactly this representation. 
\end{rem}

We will illustrate how the first and the second order Edgeworth expansion can be computed explicitly once we have $\mu^{(4)}(0), \mu^{(3)}(0), Z''(0)$ and $Z'(0)$ in terms of asymptotic moments of $S_n$. 
Because $A_0(t)= 1$ we have $R_0(t)=1$. 
From the derivation of \eqref{PolyComp} we have, 
\begin{align*}
A_1(t)=(\log\mu)^{(3)}(0)\frac{t^3}{6}-Z'(0)t &=(\mu^{(3)}(0)-3\mu''(0)\mu'(0)+2\mu'(0)^3)\frac{t^3}{6}-Z'(0)t \\ &=\big(\mu^{(3)}(0)+iA(3\sigma^2+A^2)\big)\frac{t^3}{6}-Z'(0)t \\ &= (a_{3,1}-3\sigma^2a_{1,0})\frac{(it)^3}{6}-a_{1,0}(it).
\end{align*} 
After taking the inverse Fourier transform as shown in \eqref{PolyForDensity} we have,
\begin{align*}
R_1(x)=\frac{(a_{3,1}-3\sigma^2a_{1,0})}{6\sigma^6}x(3\sigma^2-x^2)+\frac{a_{1,0}}{\sigma^2}x.
\end{align*}
Using \eqref{EdgePolyRel} we obtain the first Edgeworth polynomial,
\begin{align*}P_1(x)&=\frac{\big(a_{3,1}-3\sigma^2a_{1,0}\big)}{6\sigma^4}(\sigma^2-x^2)-\frac{a_{1,0}}{\sigma}.
\end{align*}
Similar calculations give us, 
\begin{align*}
A_2(t)&=(a_{3,1}+3\sigma^2a_{1,0})^2\frac{(it)^6}{72}+ \Big[A^2(6\sigma^2+A^4)+4a_{3,1}(A-2a_{1,0})\\&\phantom{aaaaa}-3\sigma^2(2a_{2,0} -4Aa_{1,0}+\sigma^2)+a_{4,1}\Big]\frac{(it)^4}{24} +(2a^2_{1,0}-a_{2,0})\frac{(it)^2}{2}.
\end{align*}
From \eqref{PolyForDensity} and \eqref{EdgePolyRel} we have,
\begin{align*}
R_2(t)=&(a_{3,1}+3\sigma^2a_{1,0})^2\frac{x^6-15\sigma^2x^4+45\sigma^4x^2-15\sigma^6}{72\sigma^{12}}\\ &\phantom{a}+\Big[A^2(6\sigma^2+A^4)+4a_{3,1}(A-2a_{1,0})-3\sigma^2(2a_{2,0} -4Aa_{1,0}+\sigma^2)+a_{4,1}\Big]\\ &\phantom{aaaaaaaaaaaaaaaaaaaa}\times\frac{(x^4-6\sigma^2x^2+3\sigma^2)}{24\sigma^8} +(2a^2_{1,0}-a_{2,0})\frac{(x^2-\sigma^2)}{2\sigma^4},
\end{align*}
\begin{align*}
P_2(t)=&(a_{3,1}+3\sigma^2a_{1,0})^2\frac{x(15\sigma^2-10\sigma^2x^2+x^6)}{72\sigma^{10}}\\&\phantom{a}+\Big[A^2(6\sigma^2+A^4)+4a_{3,1}(A-2a_{1,0})-3\sigma^2(2a_{2,0} -4Aa_{1,0}+\sigma^2)+a_{4,1}\Big]\\&\phantom{aaaaaaaaaaaaaaaaaaaaaaaaaaaaaa}\times \frac{x(3\sigma^2-x^2)}{24\sigma^6} +(2a^2_{1,0}-a_{2,0})\frac{x}{2\sigma^2}.
\end{align*}

\begin{rem}
Once we have $R_p$ for $p \in \mathbb{N}_0$ and $P_p$ for $p \in \mathbb{N}$, the polynomials $P_{p,g}, P_{p,d}$ and $P_{p,a}$ are given by $P_{p,g}=P_{p,d}=R_p$ and $P_{p,a}=P_p$. These relations were obtained in the proofs in section \ref{proofs}.  
\end{rem}

Also, one can compute $P_{p,l}$ using \eqref{WeakLocalPoly}: 
$$P_{p,l}(x)= \sum_{l+j=2p} \frac{(-ix)^j}{j!}\int t^j A_l(t) e^{-\frac{\sigma^2 t^2}{2}}\, dt.$$ 
For example, 
$$P_{0,l}(x)= \int  A_0(t) e^{-\frac{\sigma^2 t^2}{2}}\, dt = \sqrt{\frac{2\pi}{\sigma^2}}.$$
\begin{align*}
P_{1,l}(x)=& \int A_2(t)e^{-\frac{\sigma^2 t^2}{2}}\, dt\, -ix\int tA_1(t) e^{-\frac{\sigma^2 t^2}{2}}\, dt -\frac{x^2}{2}\int t^2 A_0(t) e^{-\frac{\sigma^2 t^2}{2}}\, dt\\
\frac{P_{1,l}(x)}{\sqrt{2\pi}}=& (a_{3,1}+3\sigma^2a_{1,0})^2\frac{5}{24\sigma^7}\\&+ \Big[A^2(6\sigma^2+A^4)+4a_{3,1}(A-2a_{1,0})-3\sigma^2(2a_{2,0} -4Aa_{1,0}+\sigma^2)+a_{4,1}\Big]\frac{1}{8\sigma^5} \\ &\phantom{aaaaaaaaaaaaaa}-(2a^2_{1,0}-a_{2,0})\frac{1}{2\sigma^6}-\bigg((a_{3,1}-3\sigma^2a_{1,0})\frac{1}{\sigma^5}+\frac{2a_{1,0}}{\sigma^3}\bigg)\frac{x}{2}-\frac{x^2}{2\sigma^3}
\end{align*}
Higher order Edgeworth polynomials can be computed similarly.

We can compare our results with the centered i.i.d.\hspace{3pt}case. Then, we have that $A=0$, $a_{1,0}=0$ because the sequence is stationary. Also, $a_{3,1}=\lim_{n \to \infty} \frac{1}{n}\EXP([S_n-nA]^3)=\EXP((X_1-A)^3)$, $a_{2,0}=0$ and $a_{4,1}=\EXP(X^4_1)$. So, the above polynomials reduce to,
\begin{align*}
A_1(t)&=\frac{\EXP(X_1^3)}{6}(it)^3,\ R_1(x)=\frac{\EXP(X_1^3)}{6\sigma^6}x(3\sigma^2-x^2),\  P_1(x)=\frac{\EXP(X_1^3)}{6\sigma^4}(\sigma^2-x^2)\\ 
A_2(t)&=\EXP(X_1^3)^2\frac{(it)^6}{72}+ (\EXP(X^4_1)-3\sigma^4)\frac{(it)^4}{24} \\
\frac{P_{0,l}(x)}{\sqrt{2\pi}}&=\frac{1}{\sigma},\ \frac{P_{1,l}(x)}{\sqrt{2\pi}}=\frac{\EXP(X^3_1)^2}{\sigma^7} \frac{5}{24} + \Big(\frac{\EXP(X^4_1)}{\sigma^5}-\frac{3}{\sigma}\Big)\frac{1}{8}- \frac{\EXP(X^3_1)}{\sigma^5}\frac{x}{2} - \frac{1}{\sigma^3} \frac{x^2}{2}
\end{align*}
These agree with the polynomials found in \cite[Chapter XVI]{Feller2} (to see this one has to replace $x$ by $x/\sigma$ to make up for not normalizing by $\sigma$ here) and \cite{Br}. The polynomials $Q_k$ found in the latter are related to $P_{k,l}$ by $Q_k(x)=\frac{1}{2\pi}P_{k,l}(x)$.  

It is also easy to see that these agree with previous work on non-i.i.d.\hspace{3pt}examples. In both \cite{CP, HP} only the first order Edgeworth polynomial is given explicitly. In \cite{CP}, because the sequence is stationary and centered, we can take $A=0$ and $a_{1,0}=0$. Also, the pressure $P(t)$ given there, corresponds to $\log \mu(t)$ in our paper. So we recover $A_1(t)=P'''(0)\frac{(it)^3}{6}$ in \cite[Theorem 3]{CP}. In \cite{HP}, sequence is centered but not assumed to be stationary. So $A=0$ and $a_{1,0} \neq 0$ and the asymptotic bias appears in the expansion and $A_1(t)=i\mu^{(3)}(0)\frac{(it)^3}{6}-a_{1,0}(it)$ which agrees with \cite[Theorem 8.1]{HP}. This dependence on initial distribution corresponds to presence of $\ell$ in \eqref{MainAssum}.


\section{Applications}\label{App}
\subsection{Local Limit Theorem} Existence of the Edgeworth expansion allows us to derive Local Limit Theorems (LLTs). For example see \cite[Theorem 4]{DF}. Also, as direct consequences of weak global Edgeworth expansions, an LCLT comparable to the one given in \cite[Chapter II]{HH}, holds. In fact, a stronger version of LCLT holds true in special cases. 

To make the notation simpler, we assume that the asymptotic mean of $S_N$ is $0$. That is $A=\lim_{N\to \infty}\EXP\big(\frac{S_N}{N}\big)=0$.

\begin{prop}\label{LCLT}
Suppose that $S_N$ satisfies the weak global Edgeworth expansion of order $0$ for an integrable function $f \in (\cF,\|\cdot\|)$ where $\|\cdot\|$ is translation invariant. Further, assume that $|xf(x)|$ is integrable. Then,
\begin{equation}
\sqrt{N}\EXP(f(S_N-u))= \frac{1}{\sqrt{2\pi \sigma^2}}e^{-\frac{u^2}{2N\sigma^2}}\int f(x)\, dx+o(1) 
\end{equation} 
uniformly for $u \in \reals$. 
\end{prop}
\begin{proof}
After the change of variables $z\sqrt{N} \to z$ in the RHS of the weak global Edgeworth expansion,
\begin{align*}
\sqrt{N}\EXP&(f(S_N-u)) \\ &= \int \fn\Big(\frac{z}{\sqrt{N}}\Big) f(z-u) dz+ \|f\|o(1)  \\ &= \int \left[\fn\Big(\frac{u}{\sqrt{N}}\Big) + (z-u) \fn'\Big(\frac{z_u}{\sqrt{N}}\Big)\right]f(z-u) dz +\|f\|o(1) 
\\ &= \fn\Big(\frac{u}{\sqrt{N}}\Big)\int f(z-u)\, dz\,  + \frac{C}{N}\int (z-u) \fn\Big(\frac{z_u}{\sqrt{N}}\Big)f(z-u) dz \hspace{2pt}+\|f\|o(1)
\end{align*}
Here $z_u$ is between $u$ and $z$ and depends continuously on $u$. 

Notice that,
$$\Big|\int (z-u) \fn\Big(\frac{z_u}{\sqrt{N}}\Big)f(z-u) dz\Big|\leq \int |(z-u)f(z-u)|dz \leq \|xf\|_1$$
Therefore, after a change of variables $z-u \to z$ in the RHS,
\begin{align*}
\sqrt{N}\EXP(f(S_N-u)) = \fn\Big(\frac{u}{\sqrt{N}}\Big)\int f(z) dz +\max\{\|xf\|_1, \|f\|\}\hspace{2pt}o(1)
\end{align*} 
as required. 
\end{proof}

In particular, the result holds for $\cF=F^1_0$. If the order $0$ weak global Edgeworth expansion holds for all $f \in F^1_0$, then we have the following corollary. We note that this is indeed the case for faster decaying $|\EXP(e^{itS_N})|$ as in Markov chains and piecewise expanding maps described in sections \ref{density}, \ref{nodensity} and \ref{EM1d}. 

\begin{cor}
Suppose that $S_N$ admits the weak global Edgeworth expansion of order $0$ for all $f \in F^1_0$. Then, for all $a<b$, 
$$\frac{\sqrt{N}}{(b-a)}\Prob\Big(S_N \in (u+a,u+b)\Big) =  \frac{1}{\sqrt{2\pi \sigma^2}}e^{-\frac{u^2}{2N\sigma^2}} +o(1)$$
uniformly in $u \in \reals$.
\end{cor}

\begin{proof}
Fix $a<b$. It is elementary to see that there exists a sequence $f_k \in F^1_0$ with compact support such that $f_k \to 1_{(u+a,u+b)}$ point-wise and $f_k$'s are uniformly bounded in $F^1_1$. This bound can be chosen uniformly in $u$, call it $C$. 

Therefore, from the proof of \Cref{LCLT}, we have,
$$\sqrt{N}\EXP(f_k(S_N-u)) = \fn\Big(\frac{u}{\sqrt{N}}\Big)\int f_k(z) dz +C^1_1(f_k)\hspace{2pt}o(1)$$ 
Because $0 \leq C^1_1(f_k) \leq C$, taking the limit as $k \to \infty$ we conclude, 
$$\sqrt{N} \Prob\Big(S_N \in (u+a,u+b)\Big)= \fn\Big(\frac{u}{\sqrt{N}}\Big)\int_{u+a}^{u+b} 1\, dz +C\hspace{2pt}o(1)$$
and the result follows. 
\end{proof}

In fact, $u$ in the previous theorem need not be fixed. For example, for a sequence $u_N$ with $\frac{u_N}{\sqrt{N}} \to u$, we have the following:

\begin{cor}
Suppose that $S_N$ admits the weak global Edgeworth expansion of order $0$ for all $f \in F^1_0$. Let $u_N$ be a sequence such that $\lim_{N \to \infty}\frac{u_N}{\sqrt{N}} = u$. Then, for all $a<b$, 
$$\lim_{N \to \infty}\frac{\sqrt{N}}{(b-a)}\Prob\Big(S_N \in (u_N+a,u_N+b)\Big) =  \frac{1}{\sqrt{2\pi \sigma^2}}e^{-\frac{u^2}{2\sigma^2}}.$$
\end{cor}

Now, we state the stronger version of LCLT in which we allow intervals to shrink. 
\begin{defin}\label{SLCLT} Given a sequence $\epsilon_N$ in $\reals^+$ with $\epsilon_N \to 0$ as $N\to \infty$, we say that $S_N$ admits an LCLT for $\epsilon_N$ if we have, 
$$\frac{\sqrt{N}}{2\epsilon_N}\Prob\Big(S_N\in (u-\epsilon_N,u+\epsilon_N)\Big)=\frac{1}{\sqrt{2\pi \sigma^2}}e^{-\frac{u^2}{2N\sigma^2}}+o(1) $$
uniformly in $u \in \reals$. 
\end{defin}

The next proposition gives a existence of weak global Edgeworth expansions as a sufficient condition for $S_N$ to admit a LCLT for a sequence $\epsilon_N$. Notice that existence of higher order expansions allow $\epsilon_N$ to decay faster. In case expansions of all orders exist, $\epsilon_N$ can decay at any subexponential rate. 

\begin{prop}
Suppose that $S_N$ satisfies the weak global Edgeworth expansion of order $r\ (\geq 1)$ for all $f\in F^1_0$. Let $\epsilon_N$ be a sequence of positive real numbers such that $\epsilon_N \to 0$ and $\epsilon_N N^{r/2}\to \infty$ as $N \to\ \infty$. Then, $S_N$ admits an LCLT for $\epsilon_N$. 
\end{prop}

\begin{proof}
WLOG assume $\epsilon_N <1$ for all $N$. As in the previous proof, there exists a sequence $f_k \in F^1_0$ with compact support such that $f_k \to 1_{(u-\epsilon_N,u+\epsilon_N)}$ point-wise and $f_k$'s are uniformly bounded in $F^1_0$. This bound can be chosen uniformly in $N$ and $u$, call it $C$. 

Let $N\in \naturals$. Note that for all $k$,
$$ \EXP(f_k(S_N))= \sum_{p=0}^r \frac{1} {N^{\frac{p}{2}}} \int P_{p,g}(z) \fn(z) 
f_k\big(z\sqrt{N}\big) dz+C^1_0(f_k)\hspace{2pt}o\left(N^{-(r+1)/2}\right).$$
By taking the limit as $k \to \infty$ and using the fact $0\leq C^1_0(f_k)\leq C$, we conclude, 
\begin{align*}
\Prob\Big(S_N \in (u-\epsilon_N,u+\epsilon_N)\Big) &= \sum_{p=0}^r \frac{1} {N^{\frac{p}{2}}} \int_{\frac{u-\epsilon_N}{\sqrt{N}}}^{\frac{u+\epsilon_N}{\sqrt{N}}} P_{p,g}(z) \fn(z)\, dz +C\hspace{2pt}o\left(N^{-(r+1)/2}\right).
\end{align*}
After a change of variables $z \to \frac{z}{\sqrt{N}}$ in the $p=0$ term and divide the whole equation by $2\epsilon_N$ to get,
\begin{align*}
\frac{\sqrt{N}}{2\epsilon_N}\Prob\Big(S_N \in (u-\epsilon_N,u+\epsilon_N)\Big) = &\frac{1}{2\epsilon_N}\int 1_{J_N}(z-u)\fn\big(\frac{z}{\sqrt{N}}\big)\, dz \\ &+\sum_{p=1}^r \frac{\sqrt{N}} {2\epsilon_N N^{\frac{p}{2}}} \int_{\frac{u-\epsilon_N}{\sqrt{N}}}^{\frac{u+\epsilon_N}{\sqrt{N}}} P_{p,g}(z) \fn(z)\, dz +C\hspace{2pt}o\left(\frac{1}{\epsilon_NN^{r/2}}\right)
\end{align*}
where $J_N=(-\epsilon_N,\epsilon_N)$.

Note that for $p\geq 1$, there exists $C_p$ such that $| P_{p,g}(z) \fn(z)|<C_p$. Therefore,
\begin{align*}
\bigg|\frac{\sqrt{N}} {2\epsilon_N N^{\frac{p}{2}}} \int_{\frac{u-\epsilon_N}{\sqrt{N}}}^{\frac{u+\epsilon_N}{\sqrt{N}}} P_{p,g}(z) \fn(z)\, dz \bigg| \leq \frac{C_p\sqrt{N}} {2\epsilon_N N^{\frac{p}{2}}} \int_{\frac{u-\epsilon_N}{\sqrt{N}}}^{\frac{u+\epsilon_N}{\sqrt{N}}} 1\, dz \leq \frac{C_p}{N^{p/2}}=o(1)
\end{align*}

Also, as in the proof of \Cref{LCLT}, 
\begin{align*}
\frac{1}{2\epsilon_N}\int 1_{J_N}(z-u)\fn\big(\frac{z}{\sqrt{N}}\big)\, dz  &= \frac{1}{2\epsilon_N}\fn\Big(\frac{u}{\sqrt{N}}\Big)\int_{u-\epsilon_N}^{u+\epsilon_N}1\, dz\, \\ &\phantom{aaaaaaaaaaa} + \frac{C}{2\epsilon_NN}\int_{u-\epsilon_N}^{u+\epsilon_N}(z-u) \fn\Big(\frac{z_u}{\sqrt{N}}\Big)\, dz 
\end{align*}
Note that,
\begin{align*}
\bigg|\frac{C}{2\epsilon_NN}\int_{u-\epsilon_N}^{u+\epsilon_N}(z-u) \fn\Big(\frac{z_u}{\sqrt{N}}\Big)\, dz\bigg| \leq \frac{C}{2\epsilon_NN}\int_{u-\epsilon_N}^{u+\epsilon_N}|z-u|\, dz= \frac{C\epsilon_N}{2N}
\end{align*}
Therefore, 
$$\frac{1}{2\epsilon_N}\int 1_{J_N}(z-u)\fn\big(\frac{z}{\sqrt{N}}\big)\, dz = \fn\Big(\frac{u}{\sqrt{N}}\Big)+o(1).$$

Combining these estimates with $\epsilon_NN^{r/2}\to \infty$ we have that, 
$$\frac{\sqrt{N}}{2\epsilon_N}\Prob\Big(S_N \in (u-\epsilon_N,u+\epsilon_N)\Big) =\fn\Big(\frac{u}{\sqrt{N}}\Big)+o(1)$$
and it is straightforward from the proof that this is uniform. 
\end{proof}
\begin{rem}
We note that this result implies \cite[Theorem 4]{DF} because existence of classical Edgeworth expansions imply the existence of the weak global Edgeworth expansion and this result is uniform in $u$.
\end{rem}
\subsection{Moderate Deviations}
While the CLT describes the typical behaviour or ordinary deviations from the mean provided by the law of large numbers, it is not sufficient to understand properties of distribution of $X_n$ completely. Therefore, the study of excessive deviations is important. 

For example, deviations of order $n$ are called large deviations. An exponential moment condition is required for a large deviation principle to hold, even for the i.i.d.\hspace{3pt}case. However, when deviations are of order $\sqrt{n\log n}$ (moderate deviations) this is not the case. We show here that a moderate deviation principle holds for $S_N$ under a weaker assumption than the exponential moment assumption. 

It is also worth noting that moderate deviations have numerous applications in areas like statistical physics and risk analysis. For example, moderate deviations are greatly involved in the computation of Bayes risk efficiency. See \cite{RS1} for details.  

\begin{prop}\label{ModDevThm}
Suppose $S_N$ admits the order $r$ Edgeworth expansion. Then for all $c \in (0,r)$, when $1 \leq x \leq \sqrt{c \sigma^2\ln N},$
\begin{align}\label{ModDev}
\lim_{N\to \infty} \frac{1-\Prob\Big(\frac{S_N -AN}{\sqrt{N}}\leq x\Big)}{1-\fN(x)}=1.
\end{align} 
\end{prop}
\begin{proof}
Note that, 
\begin{align*}
1-\fN(x) - \Big[1-\Prob\Big(\frac{S_N -AN}{\sqrt{N}} \leq x\Big)\Big] &=\Prob\Big(\frac{S_N -AN}{\sqrt{N}} \leq x\Big)- \fN(x) \\ &=\sum_{p=1}^r \frac{P_p(x)}{N^{p/2}} \fn(x)+o\left(N^{-r/2}\right)
\end{align*}
uniformly in $x$. So it is enough to show that for $1 \leq x \leq \sqrt{c \sigma^2\ln N}$,
$$\lim_{N \to \infty}\frac{P_p(x)\fn(x)}{N^{p/2}(1-\fN(x))}  =0\ \text{and}\ \frac{N^{-r/2}}{1-\fN(x)}=o(1)$$
Note that for $x\geq 1$, $$1-\fN(x) = \frac{\sigma^2\fn(x)}{x} + \cO\Big(\frac{\fn(x)}{x^3}\Big).$$ Thus, 
\begin{align*}
\frac{N^{-r/2}}{1-\fN(x)} \leq \frac{N^{-r/2}}{1-\fN(\sqrt{c \sigma^2\ln N})} &= \cO\Big(\sqrt{\ln{N}}\frac{N^{-r/2}}{e^{-\frac{c}{2} \ln N}} \Big) \\ &= \cO\Big(\frac{\ln N}{N^{(r-c)/2}}\Big)
\end{align*}
Say $P_p(x)$ is of degree $q$. Then for some $C$ and $K$, 
\begin{align*}
\Big|\frac{P_p(x)\fn(x)}{N^{p/2}(1-\fN(x))}\Big| \leq C\frac{(x^{q}+K)\fn(x)}{N^{p/2}(1-\fN(x))} &= C\frac{(x^{q}+K)}{N^{p/2}}x\Big(1+\cO\Big(\frac{1}{x^2}\Big)\Big) \\ &\leq C\frac{(\ln N)^{q+1}}{N^{p/2}} \to 0\ \text{as}\ N \to \infty.
\end{align*}
This completes the proof of \eqref{ModDev}. 
\end{proof}
\Cref{ModDevThm} is a generalization of the results on moderate deviations found in \cite{RS} to the non-i.i.d.\hspace{3pt}case along with improvements on the moment condition. It should be noted that \cite{Br} contains an improvement of the moment condition for the i.i.d.\hspace{3pt}case. But the proof we present here is different from the proof presented in \cite{Br}. 

As an immediate corollary to the above theorem, we can state the following first order asymptotic for probability of moderate deviations. 
\begin{cor}
Assume $S_N$ admits the order $r$ Edgeworth expansion. Then for all $c \in (0,r)$,
$$\Prob(S_N  \geq AN+\sqrt{c \sigma^2N\ln N}) \sim  \frac{1}{ \sqrt{2\pi c }} \frac{1}{\sqrt{N^c\ln N}}.$$
\end{cor}

\section{Examples}\label{Examples}
Here we give several examples of systems satisfying assumptions (A1)--(A4).

\subsection{Independent variables.}\label{IID}
Let $X_n$ be i.i.d.\hspace{3pt}with $r+2$ moments. In this case we can take $\Ban=\reals,$ and define $\cL_tv=\EXP(e^{itX_1}v)=\phi(t)v$ where $\phi$ is the characteristic function of $X_1$. Here we have taken $\ell=1$. Put $v=1$. Then, the independence of the random variables gives us, $\cL_t^n1=\EXP(e^{itS_n})=\phi(t)^n$.  Also, the moment condition implies $t \to \phi(t)$ is $C^{r+2}$. This means (A1) is satisfied. (A2) is clear.

Suppose $X_1$ is $l-$Diophantine. That is there exists $C>0$ and $x_0>0$ such that for all $|x|>x_0$, $|\phi(t)|<1-\frac{C}{|x|^l}$. Then $|\phi(t)| \leq e^{-\frac{C}{|t|^l}}$. So $|\phi(t)|<1$ for all $t\neq 0$. So we have (A3). Also, this implies that $X_1$ is non-lattice. An easy computation shows that when $r_1<\frac{1}{l}$, there exists $r_2$ such that $t_0<|t|<n^{r_1} \implies |\phi(t)|^n \leq n^{-r_2}$. In fact, $|\phi(t)|^n \leq e^{-cn^\alpha}$ where $\alpha=1-r_1l > 0$. So, (A4) is satisfied with $r_1<\frac{1}{l}$. 

When $l=0$ we see that (A4) is satisfied with $r_1>\frac{r-1}{2}$ and hence by Theorem \ref{EdgeExp} order $r$ Edgeworth expansion for $S_n$ exists. This is exactly the classical result due to Cram\'er because Cram\'er's continuity condition: $\limsup_{|t|\to\infty}|\phi(t)|<1$ corresponds to $l=0$.

Choose $q>\frac{r+1}{2r_1} > \frac{(r+1)l}{2}$. Then, by Theorem \ref{WLEdgeExp} and Theorem \ref{WGEdgeExp} we have that $S_n$ admits weak global expansion for $f\in F^{q+2}_{0}$ and weak local expansion for $f\in F^{q+2}_{r+1}$. These are similar to the results appearing in \cite{Br} but slightly weaker because we require one more derivative: $q+2 > 2+\frac{(r+1)l}{2}$ as opposed to $1+\frac{(r+1)l}{2}$. This is because we do not use the optimal conditions for the integrability of the Fourier transform. If we required $f \in F^{q+1}_r$ and $f^{(q+1)}$ to be $\alpha-$H\"older for small $\alpha$, then the proof would still hold true and we could recover the results in \cite{Br}. 

\subsection{Finite state Markov chains}\label{FiniteMC}
Here we present a non-trivial example for which the weak Edgeworth expansions exist but the strong expansion does not exist. 

Consider the Markov chain $x_n$ with states $S=\{1,\dots,d\}$ whose transition probability matrix $P=(p_{jk})_{d\times d}$ is positive. Then, by the Perron-Forbenius theorem, $1$ is a simple eigenvalue of $P$ and all other eigenvalues are strictly contained inside the unit disk. Suppose $\bh=(h_{jk})_{d\times d} \in $ M$(d,\reals)$ and its entries cannot be written as $$r h_{jk}=c+H(k)-H(j) \mod\ 2\pi$$ for some $d-$vector $H$ and $r \in \reals$. Put $X_n=h_{x_nx_{n+1}}$. 

For the family of operators $\cL_t : \complex^d \to \complex^d$,
\begin{equation}\label{TransProb}
(\cL_tf)_j = \sum_{k=1}^d e^{ith_{jk}}p_{jk}f_k,\ j=1,\dots,d
\end{equation}
$v=1$ and $\ell=\mu_0$, the initial measure, we have \eqref{MainAssum}. 

Define $b_{r,j,k}=h_{rj}+h_{jk}$ for all $j,r=1,\dots,d$ and $k=2,\dots,d$. Put $d(s)=\max\ \{(b_{r,j,k}-b_{r,1,k})s\}$ where $\{\ . \ \}$ denotes the fractional part. We further assume that $\bh$ is $\beta-$Diophantine, that is, there exists $K \in \reals$ such that for all $|s|>1$, 
\begin{equation}\label{DiophCon}
d(s)\geq \frac{K}{|s|^\beta}.
\end{equation}
If $\beta>\frac{1}{d^2(d-1)-1}$ then almost all $\bh$ are $\beta-$Diophantine. 

Because $S_n$ can take at most $\cO(n^{d^2-1})$ distinct values, $S_n$ has a maximal jump of order at least $n^{-(d^2-1)}$. Therefore, the process $X^{\bh}_n=h_{x_nx_{n-1}}$ does not admit the order $2(d^2-1)$ Edgeworth expansion. 


The Perron-Forbenius theorem implies that the operator $\cL_0$ satisfies (A2). Because \eqref{TransProb} is a finite sum, it is clear that $t \mapsto \cL_t$ is analytic on $\reals$. So we also have (A1). Also the spectral radius of $\cL_t$ is at most $1$. Assume $\cL_t$ has an eigenvalue on the unit circle, say $e^{i\lambda}$, then, 
\begin{align*}
e^{i\lambda}f_j=(\cL_tf)_j = \sum_{k=1}^d e^{ith_{jk}}p_{jk}f_k
\end{align*}
Assuming $\max_j |f_j| = |f_r| $,
\begin{align*}
|f_r| &= |e^{i\lambda}f_r| = \bigg|\sum_{k=1}^d e^{ith_{jk}}p_{jk}f_k \bigg| \leq   \sum_{k=1}^d p_{jk}|f_k|\implies   \sum_{k=1}^d p_{jk}(|f_k|-|f_r|)\geq 0 
\end{align*} 
Because $|f_k|-|f_r| \leq 0$ for all $k$ and $p_{jk}\geq 0$ for all $j$ and $k$ we have $|f_k|=|f_r|$ for all $k$. Write $f_k=Re^{iH(k)}$ for all $k$. Then, 
\begin{align*}
e^{i\lambda}Re^{iH(j)} &= \sum_{k=1}^d e^{ith_{jk}}p_{jk}Re^{iH(k)}\\
0&=\sum_{k=1}^d p_{jk}(e^{i(th_{jk}+H(k)-H(j)-\lambda)}-1) \\ 
\implies th_{jk}&=\lambda+H(j)-H(k) \mod\ 2\pi
\end{align*}
But this is a contradiction. Therefore, (A3) holds. Next we notice that, 
\begin{align}\label{NormEst}
|(\cL^2_tf)_r| = \bigg|\sum_{j=1}^d\sum_{k=1}^d e^{it(h_{rj}+h_{jk})}p_{rj}p_{jk}f_k \bigg| \nonumber &= \bigg|\sum_{k=1}^d \bigg(\sum_{j=1}^d e^{it(h_{rj}+h_{jk})}p_{rj}p_{jk} \bigg) f_k \bigg| \nonumber \\ &\leq  \|f\| \bigg(\sum_{k=1}^d \bigg| \sum_{j=1}^d e^{itb_{r,j,k}}p_{rj}p_{jk}  \bigg|\bigg)
\end{align}
Now we estimate $|b_{r,k}(t)|$ where
\begin{align*}
b_{r,k}(t)=\sum_{j=1}^d e^{itb_{r,j,k}}p_{rj}p_{jk} =e^{itb_{r,1,k}}\sum_{j=1}^d e^{it(b_{r,j,k}-b_{r,1,k})}p_{rj}p_{jk}
\end{align*}
Then we have,
\begin{align*}
|b_{r,k}(t)|^2&=\sum_{j=1}^d p^2_{rj}p^2_{jk}+2\sum_{j>l}^d p_{rj}p_{jk}p_{rl}p_{lk}\cos((b_{r,j,k}-b_{r,l,k})t)\\ &=\bigg(\sum_{j=1}^d p_{rj}p_{jk}\bigg)^2-2\sum_{j>l}^d p_{rj}p_{jk}p_{rl}p_{lk}[1-\cos((b_{r,j,k}-b_{r,l,k})t)] \\ &=\bigg(\sum_{j=1}^d p_{rj}p_{jk}\bigg)^2-2Cd(t)^2+\cO(d(t)^3),\ C>0\\  |b_{r,k}(t)| &=\sum_{j=1}^d p_{rj}p_{jk}-\tilde{C}d(t)^2+\cO(d(t)^3),\ \tilde{C}>0
\end{align*}
Therefore, 
\begin{align*}
\sum_{k=1}^d \bigg| \sum_{j=1}^d e^{itb_{r,j,k}}p_{rj}p_{jk}  \bigg| &= \sum_{k=1}^d \bigg(\sum_{j=1}^d p_{rj}p_{jk}\bigg)-\overline{C}d(t)^2+\cO(d(t)^3) \\ &= 1 -\overline{C}d(t)^2+\cO(d(t)^3),\ \overline{C}>0
\end{align*}
From the Diophantine condition \eqref{DiophCon}, we can conclude that there exists $\theta >0$ such that for all $|t|>1$,
$$
\|\cL^2_t\| \leq 1- \theta d(t)^2 \implies \|\cL^N_t\|  \leq \big(1- \theta d(t)^2  \big)^{\lceil N/2 \rceil} \leq e^{-\theta  d(t)^2  N/2}\leq e^{-\theta  t^{-2\beta} N/2}.
$$

When $1<|t|<N^{\frac{1-\epsilon}{2\beta}}$, we have, $\|\cL^N_t\|  \leq e^{-\theta N^{\epsilon}/2}$
which gives us (A4) with $r_1=\frac{1-\epsilon}{2\beta}$ where $\epsilon>0$ can be made as small as required. Because for small $\epsilon$, $\lceil \frac{r+1}{2(1-\epsilon)} \rceil= \lceil \frac{r+1}{2} \rceil$, choosing $q>\frac{r+1}{2}\beta$, we conclude that for $f \in F^{q+2}_0$ weak global and for $f \in F^{q+2}_{r+1}$ weak local Edgeworth expansions of order $r$ for the process $X^{\bh}_n$ exist. 

Also, $S_N$ admits  averaged Edgeworth expansions of order $r$ for $f \in F^2_0$. In the special case of $\beta > \frac{1}{d^2(d-1)-1}$, these hold for a full measure set of $\bh$ even though the order $r$ strong expansion does not exist for $r +1 \geq d^2$. 


\subsection{More general Markov chains}
\subsubsection{Chains with smooth transition density}\label{density}
First we consider the case where $x_n$ is a time homogeneous Markov process on a compact connected manifold $\cM$ with smooth transition density $p(x,y)$ which is bounded away from $0$, and $X_n=h(x_{n-1}, x_n)$ for a piece-wise smooth function $h:\cM\times \cM\to\reals.$  
We also assume that,
\begin{equation} 
\label{MC-CoB}
\nexists H \in L^{\infty}(\cM)\ \text{s.t.}\ h(x,y)+H(y)-H(x)\ \text{is piecewise constant}.
\end{equation}

It is well known that when there does not exist $H$ such that $h(x,y)+H(y)-H(x)$ is constant and the transition probability has a non-degenrate absolute continuous component then the CLT holds with $\sigma^2>0$. 

To check the assumption \ref{MC-CoB} we need the following:
\begin{lem}
\label{LmMc-CoB}
\eqref{MC-CoB} does not hold iff there exists $o\in\cM$ such that the function $x \mapsto h(o,x)+h(x,y)$ is  piece-wise constant.
\end{lem}
\begin{proof}
If \eqref{MC-CoB} does not hold then for each $o\in\cM$
$$h(o,x)+h(x,y)=c_1(o,x)+c_2(x,y)+H(y)-H(o)$$ 
where $c_i$, $i=1,2$ are piece-wise constant in $x$. 

Conversely, if for some $o\in \cM,$ $x \mapsto h(o,x)+h(x,y)$ is piece-wise constant. 
Let $c=h(o,o)$ and $H(x)=h(o,o)-h(o,x)$. 
Then, $h(o,o)+h(o,y)$ and $h(o,x)+h(x,y)$ differ by a piece-wise constant function. Then \eqref{MC-CoB} does not hold because $h(o,x)+h(x,y)-(h(o,o)+h(o,y))=h(x,y)+H(y)-H(x)-c$ is piecewise constant.
\end{proof}
Let $\Ban=L^{\infty}(\cM)$ and consider the family of integral operators,
$$(\cL_t u)(x)=\int p(x,y) e^{it h(x,y)} u(y)\, dy. $$
Let $\mu$ be the initial distribution of the Markov chain and $\{\cF_n\}$ be the filtration adapted to the processes. Then, using the Markov property,
\begin{align*}
\EXP_{\mu}[e^{itS_n}]=\EXP_{\mu}[e^{itS_{n-1}} \cL_t\scalebox{1.10}{$1$}].
\end{align*}
By induction we can conclude $$\EXP_{\mu}(e^{itS_n})=\int \cL^n_t \scalebox{1.10}{$1$} \, d\mu$$
Because $h$ is bounded, expanding $e^{ith(x,y)}$ as a power series in $t$, we see that $t \mapsto \cL_t$ is analytic for all $t$. This shows that (A1) is statisfied. 

One can uniformly approximate $p(x,y)$ by functions of the form $\sum_{k=1}^n q_k(x)r_k(y)$. Therefore, $\cL_t$ is a uniform limit of finite rank operators and is compact. Compact operators have a point spectrum hence the essential spectral radius of $\cL_t$ vanishes. It is also immediate that $\|\cL_t\| \leq 1$ for all $t$. Hence the spectrum is contained in the closed unit disk.

In addition, $\cL_0 : L^{\infty}(\cM) \to L^{\infty}(\cM)$ given by $$(\cL_0 u)(x)=\int p(x,y) u(y)\, dy$$ is a positive operator. Note that $(\cL_0 1)(x)=1$ for all $x$. Thus, $1$ is an eigenvalue of $\cL_0$ with eigenfunction $1$. Also, eigenvalue $1$ is simple and all other eigenvalues $\beta$ are such that $|\beta|<1$. This follows from a direct application of Birkhoff Theory (see \cite{Bi}). Thus, we have (A2). 

Next we show that if $\beta \in$ sp$(\cL_t)$, $t\neq 0$ then $|\beta|<1$. If not, then there exists $\lambda$ and $u \in L^\infty (\cM)$ such that $$ \int p(x,y)e^{ith(x,y)}u(y)\, dy = e^{i\lambda}u(x)$$
Suppose $\sup_{x} |u(x)|=R$ then for each $\epsilon>0$ there exists $x_\epsilon$ such that 
\begin{align*}
R-\epsilon \leq |u(x_\epsilon)|= |e^{i\lambda}u(x_\epsilon)| &= \left|\int p(x,y)e^{ith(x,y)}u(y)\, dy \right| \leq \int p(x,y) |u(y)| \, dy 
\end{align*}
Therefore,
$$\int p(x,y)[|u(y)|-R] \, dy \geq - \epsilon,$$
But $|u(y)|-R \leq 0$. Hence, $|u(y)|=R$ a.e.  
Therefore, $u(y)=Re^{i\theta(y)}$ a.e. for some function $\theta$ and we may assume $\theta \in [0,2\pi)$.
\begin{gather}
\int p(x,y)e^{ith(x,y)}Re^{i\theta(y)}\, dy = Re^{i\lambda}e^{i\theta(x)} \nonumber \\ 
\implies \int p(x,y)[e^{i\left(th(x,y)-\lambda +\theta(y)-\theta(x)\right)}-1]\, dy =0 \nonumber \\ \implies th(x,y)-\lambda +\theta(y)-\theta(x) \equiv 0\mod 2\pi \label{hNonLattice}
\end{gather}
Thus, $x \mapsto h(y,x)+h(x,z)$ does not depend on $x$ modulo $2\pi$ i.e.\hspace{3pt}it is piece-wise constant. 
By Lemma \ref{LmMc-CoB}, $h(x,y)$ does not satisfy \eqref{MC-CoB}. This contradiction proves (A3).

Recall that if $\cK$ is integral operator 
$$(\cK u)(x)=\int k(x,y) u(y) dy$$ then $$\|\cK\|=\sup_x \int |k(x,y)| dy. $$
In our case $\cL_t^2$ has the kernel,
$$ \fl_t(x,y)=\int e^{it [h(x,z)+h(z,y)]} p(x,z) p(z,y) dz. $$ 
By Lemma \ref{LmMc-CoB} for each $x$ and $y$ the function $z\mapsto (h(x,z)+h(z,y))$ is not piecewise constant. So its derivative (whenever it exists) is not identically 0. Thus there is an open set $V_{x,y}$ and a vector field $e$ such that $\P_e  [h(x,z)+h(z,y)]\neq 0$ on $V_{x,y}$.
Integrating by parts in the direction of $e$ we conclude that
$$ \lim_{t\to\infty} \int_{V_{x,y}}  e^{it [h(x,z)+h(z,y)]} p(x,z) p(z,y) dz=0. $$
By compactness there are constants $r_0, \eps_0$ such that for $|t|\geq r_0$ and all $x$ and $y$ in $\cM$, 
$|\fl_t(x,y)|\leq \fl_0(x,y)-\eps_0.$ It follows that
$$ ||\cL_t^2||=\sup_{x,y}\int_\cM |\fl_t(x,y)| dy\leq \int_\cM \fl_0(x,y) dy\ -\eps_0. $$
The first term here equals
$$ \iint_{\cM\times \cM} p(x,z) p(z,y) dz dy=1. $$
Hence for $|t|\geq r_0,$ 
$||\cL^2_t||\leq 1-\eps_0$ and so $||\cL^N_t||\leq (1-\eps_0)^{\lceil N/2 \rceil}$.
This proves (A4) with no restriction on $r_1$. Therefore, $S_N$ admits Edgeworth expansions of all orders.  \\

Next we look at the case when \eqref{MC-CoB} fails but the constants are not lattice valued. Then, arguments for (A1), (A2) and (A3) hold. In particular, \eqref{hNonLattice} cannot hold since it implies that  $$\left(h(x,y) +\frac{\theta(y)}{t}-\frac{\theta(x)}{t}\right)\in \frac{\lambda}{t}+ \frac{2\pi}{t}\integers$$ 
However, we have to impose a Diophantine condition on the values that $h(x,y)$ can take in order to obtain a sufficient control over $\|\cL^N_t\|$ and obtain (A4).  

For fixed $x,y$ let the range of $z \mapsto h(x,z)+h(z,y)$ be $S=\{c_1,\dots, c_d\}$. Note that these $c_i$'s may depend on $x$ and $y$. However, there can be at most finitely many values that $ h(x,z)+h(z,y)$ can take as $x$ and $y$ vary on $\cM$ because $h$ is piece-wise smooth. So we might as well assume that $S$ is this complete set of values. Also, take $U_k$ to be the open set on which $z \mapsto h(x,z)+h(z,y)$ takes value $c_k$. Take $b_k=c_k-c_1$ and define $d(s)=\max\ \{b_ks\}$. Assume further that there exists $K>0$  such that for all $|s|>1$, 
$$d(s) \geq \frac{K}{|s|^\beta}$$ 
If $\beta>(d-1)^{-1}$ for almost all $d-$tuples $\bc=(c_1,\dots,c_d)$, the above holds. 

Note that,  
\begin{align*}
|\cL^2_t u(x)| &=\int \left|\int e^{it [h(x,z)+h(z,y)]} p(x,z) p(z,y)\, dz\right| |u(y)|\, dy  \\ &\leq \|u\|\int\left|\sum_{k=1}^d e^{itc_k} \int_{U_k}p(x,z) p(z,y)\, dz\right|\, dy  =\|u\|\int\left|\sum_{k=1}^d p_ke^{itb_k}\right|\, dy  
\end{align*}
where  and $p_k=\int_{U_k}p(x,z) p(z,y)\, dz$. Therefore, $p_1+\dots+p_d=p(x,y)$. 

Now the situation is similar to that of \eqref{NormEst} and a similar calculation yields,
$$\left|\sum_{k=1}^d p_ke^{itb_k}\right| =p(x,y)-Cd(t)^2+\cO(d(t)^3),\ C>0$$
Therefore,
$$\|\cL^2_t \| \leq \int \Big[p(x,y)-Cd(t)^2+\cO(d(t)^3)\Big] dy = 1 - \tilde{C} d(s)^2$$
From this we can repeat the analysis done in the finite state Markov chains example following \eqref{NormEst}. In particular, when $1<|t|<N^{\frac{1-\epsilon}{2\beta}}$, there exists $\theta>0$ such that
\begin{align*}
\|\cL^N_t\|  \leq e^{-\theta N^{\epsilon}}
\end{align*}
which gives us (A4). 

Finally, when \eqref{MC-CoB} fails and $h$ takes integer values with span $1$, $X_n$ is a lattice random variable and we can discuss the existence of the lattice Edgeworth expansion. In this case $S_N$ admits the lattice expansion of all orders. To this end, only the condition $\widetilde{(\text{A3})}$ needs to be checked. First note that $\cL_0=\cL_{2\pi k}$ for all $k \in \integers$. Also, assuming $\cL_t$ has an eigenvalue on the unit circle, we conclude \eqref{hNonLattice},
$$ th(x,y)-\lambda +\theta(y)-\theta(x) \equiv 0\mod 2\pi$$
This implies $t(h(x,y)+h(y,x)) \in 2\pi \integers + 2\lambda$. Note that LHS belongs a lattice with span $t$ and RHS is a lattice with span $2\pi$. Because $t$ is not a multiple of $2\pi$ this equality cannot happen. Therefore, when $t\not\in2\pi \integers$, sp$(\cL_t)\subset \{|z|<1\}$ and we have the claim. 
\subsubsection{Chains without densities}\label{nodensity}
We consider a more general case where transition probabilities may not have a density. We claim we can recover (A1)--(A4) if the transition operator takes the form $$\cL_0 = a \cJ_0 + (1-a)\cK_0$$ where $a \in (0,1)$ and $ \cJ_0$ and $\cK_0$ are Markov operators on $L^\infty(\cM)$ (i.e.\hspace{3pt}$\cJ_0f\geq 0$ if $f\geq 0$ and $\cJ_01=1$ and similarly for $\cK_0$), $$\cJ_0f(x) = \int p(x,y)f(y)\, d\mu(y)$$
and $$\cK_0f(x)=\int f(y)Q(x,dy)$$ 
where $p$ is a smooth transition density and $Q$ is a transition probability measure. Let $h(x,y)$ be piece-wise smooth and put,
$$\cJ_t(f)=\cJ_0(e^{ith}f)\ \ \text{and}\ \ \cK_t(f)=\cK_0(e^{ith}f).$$
Defining $\cL_t=a\cJ_t+(1-a)\cK_t$ we can conclude $t \mapsto \cL_t$ is analytic and that $$\EXP_{\mu}(e^{itS_n})=\int \cL^n_t \scalebox{1.10}{$1$} \, d\mu.$$


Now we show that conditions (A2), (A3) and (A4) are satisfied. Because $\|\cJ_t\| \leq 1 $ and $\|\cK_t\|\leq 1$ we have $\|\cL_t\| \leq 1$. Thus the spectral radius of $\cL_t$ is $\leq 1$. Because $a\cJ_t$ is compact, $\cL_t$ and $(1-a)\cK_t$ have the same essential spectrum. See \cite[Theorem IV.5.35]{Kato}. However the spectral radius of the latter is at most $(1-a)$. Hence, the essential spectral radius of $\cL_t$ is at most $(1-a)$. 

Because both $\cJ_0$ and $\cK_0$ are Markov operators we can conclude that $1$ is an eigenvalue of $\cL_0$ with constant function $1$ as the corresponding eigenfunction. From the previous paragraph the essential spectral radius of $\cL_0$ is at most $(1-a)$. Because $\cL^n$ is norm bounded it cannot have Jordan blocks. So $1$ is semisimple. 

Suppose, $\cL_tu=e^{i\theta}u$. Without loss of generality we may assume $\|u\|_\infty=1$. Assuming there exists a positive measure set $\Omega$ with $|u(x)|<1-\delta$ we can conclude that, for all $x$, 
\begin{align*}
|u(x)|=|L_tu(x)|&=\left|a\cJ_tu(x) + (1-a)\cK_tu(x)\right| 
\\& \leq a \int_{\Omega}|u(y)| p(x,y) d\mu(y) + a \int_{\Omega^c}|u(y)| p(x,y) d\mu(y) + (1-a) \\ &\leq 1-a\delta\mu(\Omega).
\end{align*}
This is a contradiction. Therefore, $|u(x)|=1$. Put $u(x)=e^{i\gamma(x)}$. Then,  
\begin{align*}
1=a\int e^{i(th(x,y)+\gamma(y)-\gamma(x)-\theta)} p(x,y) d\mu(y) + (1-a)e^{-i(\theta+\gamma(x))}\cK_t u 
\end{align*}
Hence, $\int e^{i(th(x,y)+\gamma(y)-\gamma(x)-\theta)} p(x,y) d\mu(y)=1 \implies \cJ_tu=e^{i\theta}u$. From \cref{density}, this can only be true when $t=0$ and in this case $\theta=0$ and $u\equiv 1$. This concludes that $\cL_t$, $t\neq 0$ has no eigenvalues on the unit disk and the only eigenvalue of $\cL_0$ on the unit disk is $1$ and its geometric multiplicity is $1$. As $1$ is semisimple, it is simple as required. This concludes proof of (A2) and (A3).

From the previous case, there exists $r>0$ and $\epsilon \in (0,1)$ such that such that for all $|t|>r$ we have $\|\cJ^2_t\| \leq 1-\epsilon$. From this we have, $\|\cL^2_t\| =\|a^2\cJ^2_t + a(1-a) \cJ_t\cK_t+(1-a)a\cK_t\cJ_t +(1-a)^2\cK^2_t\| \leq 1 - a^2\epsilon$.
Hence, for all $|t|>r$, for all $N$, $\|\cL^N_t\| \leq (1-a^2\epsilon)^{\lfloor N/2 \rfloor}$ which gives us (A4) with no restrictions on $r_1$. Therefore, $S_N$ admits Edgeworth expansions of all orders as before. 

As in the previous section, an analysis can be carried out when \eqref{MC-CoB} fails. The conclusions are exactly the same. 
\subsection{One dimensional piecewise expanding maps}\label{EM1d}
Here we check assumptions \eqref{MainAssum}, (A1)--(A4) for piecewise expanding maps of the interval using the results of \cite{BE, L1}. 

Let $f:[0,1]\to [0,1]$ be such that there is a finite partition $\cA_0$ of $[0,1]$ (except possibly a measure 0 set) into open intervals such that for all $I \in \cA_0$, $f |_I$ extends to a $C^2$ map on an interval containing $\overline{I}$. In other words $f$ is a piece-wise $C^2$ map. Further, assume that $f'\geq \lambda>1$ i.e.\hspace{2pt}$f$ is uniformly expanding. Next, let $\cA_n = \bigvee_{k=0}^n T^{-j} \cA_0$ and suppose for each $n$ there is $N_n$ such that for all $I \in \cA_n$, $f^{N_n}I=[0,1]$. Such maps are called \textit{covering}. 

Statistical properties of piece-wise $C^2$ covering expanding maps of an interval, are well-understood. For example, see \cite{L1}. In particular, such a function $f$ has a unique absolutely continuous invariant measure with a strictly positive density
$h\in\BV[0,1]$ and the associated transfer operator 
$$\cL_0 \vf(x)=\sum_{y\in f^{-1}(x)} \frac{\vf(y)}{f'(y)}$$
has a spectral gap. 

Let $g$ be $C^2$ except possibly at finite number of points and admitting a $C^2$ extension on each interval of smoothness. Define $X_n=g \circ f^n$ and consider it as a random variable with $x$ distributed according to some measure $\rho(x) dx$, $\rho\in\BV[0,1]$. 

Define a family of operators $\cL_t:$ BV$[0,1]\to $ BV$[0,1]$ by
$$\cL_t \vf(x)=\sum_{y\in f^{-1}(x)} \frac{ e^{itg(y)}}{f'(y)}\vf(y)$$ 
where $t=0$ corresponds to the transfer operator. Because $g$ is bounded, writing $e^{itg(y)}$ as a power series we can conclude $t \to \cL_t$ is analytic for all $t$. This gives (A1). 

(A2) follows from the fact that $\cL_0$ has a spectral gap. We further assume that 
\begin{equation}\label{NotPWConst}
g \text{ is not cohomologous to a piece-wise constant function.}
 \end{equation}
In particular, $g$ is not a BV coboundary. 

The assumption \eqref{NotPWConst} is reasonable. Indeed, suppose that $g$ is piece-wise constant 
taking values $c_1, c_2\dots c_k.$ Then $S_n$ takes less than $n^{k-1}$ distinct values so the maximal jump is of order at least $n^{-(k-1)}$ so $S_n$ can not admit Edgeworth expansion of order $(2k-2)$ in contrast to the case where \eqref{NotPWConst} holds as we shall see below.

A direct computation gives,
$$\EXP(e^{it S_n/\sqrt{n}})=\int_0^1 \cL_{t/\sqrt{n}}^n\rho(x)\, dx.$$
Therefore, there exists $A$ such that, 
\begin{equation}\label{eq:limit}
\lim_{n\to\infty} \EXP(e^{it \frac{S_n-nA}{\sqrt{n}}})=e^{-t^2\sigma^2/2}
\end{equation}
where $\sigma^2 \geq 0$. It is well know that $\sigma^2>0 \iff g$ is  a $\BV$ coboundary (see \cite{G}). From \eqref{eq:limit} it is clear that $S_n$ satisfies the CLT.   

To show (A3) holds, we first normalize the family of operators,
$$\overline{\cL}_t v(x) = \sum_{f(y)=x}\frac{e^{itg(y)}h(y)}{f'(y)h\circ f (y)} v(y)$$
Then, $\overline{\cL}_t=H^{-1} \circ \cL_t \circ H $ where $H$ is multiplication by the function $h$. Therefore, $\cL_t$ and $\overline{\cL}_t$ have the same spectrum. However, the eigenfunction corresponding to the eigenvalue $1$ of $\overline{\cL}_0$ changes to the constant function $1$. 

Assume $e^{i\theta}$ is an eigenvalue of $\overline{\cL}_t$. Then, there exists $u \in$ BV$[0,1]$ with $\overline{\cL}_tu(x)=e^{i\theta}u(x)$. Observe that,
\begin{align*}
\overline{\cL}_0|u|(x)=\sum_{f(y)=x}\frac{|u(y)|h(y)}{f'(y)h\circ f (y)}\geq \bigg|\sum_{f(y)=x}\frac{e^{itg(y)}u(y)h(y)}{f'(y)h\circ f (y)} \bigg|=|\overline{\cL}_tu(x)| =|e^{i\theta}u(x)| = |u(x)|
\end{align*}
Also note that, $\overline{\cL}_0$ is a positive operator. Hence, $\overline{\cL}^n_0|u|(x) \geq |u(x)|$ for all $n$. However, $$\lim_{n \to \infty} (\overline{\cL}^n_0|u|)(x) = \int |u(y)| \cdot 1 \, dy$$ 
because $1$ is the eigenfunction corresponding to the top eigenvalue. So for all $x$, $$\int |u(y)|\, dy \geq |u(x)|$$ This implies that $|u(x)|$ is constant. WLOG $|u(x)|\equiv 1$. So we can write $u(x)=e^{i\gamma(x)}$. Then, 
$$ \overline{\cL}_t u(x) = \sum_{f(y)=x}\frac{h(y)}{f'(y)h\circ f (y)} e^{i(tg(y)+\gamma(y))} = e^{i (\theta+\gamma(x))} $$
$$ \implies \sum_{f(y)=x}\frac{h(y)}{f'(y)h\circ f (y)} e^{i(tg(y)+\gamma(y)-\gamma(f(y))-\theta)} = 1 $$
for all $x$. Since, $$ \overline{\cL}_0 1 = \sum_{f(y)=x}\frac{h(y)}{f'(y)h\circ f (y)}  = 1$$ 
and $e^{i(tg(y)+\gamma(y)-\gamma(x)-\theta)}$ are unit vectors, it follows that 
\begin{equation}\label{LatticeVal}
tg(y)+\gamma(y)-\gamma(f(y))-\theta = 0\mod 2\pi
\end{equation}
for all $y$. Because $g$ is not cohomologous to a piecewise constant function
 we have a contradiction. Therefore, $\overline{\cL}_t$ and hence $\cL_t$ does not have an eigenvalue on the unit circle when $t\neq 0$. 

To complete the proof of (A3) one has to show that the spectral radius of $\cL_t$ is at most $1$ and that the essential spectral radius of $\cL_t$ is strictly less than $1$. This is clear from Lasota-Yorke type inequality in \cite[Lemma 1]{BE}. In fact, there is a uniform $\kappa \in (0,1)$ such that  $r_{ess}(\cL_t) \leq \kappa$ for all $t$. 

Next, we describe in detail how the estimate in \cite[Proposition 1]{BE} gives us (A4). To make the notation easier we assume $t>0$ and we replace $|t|$ by $t$. \cite[Proposition 1]{BE} implies that there exist $c$ and $C$ such that if $K_1$ large enough (we fix one such $K_1$) then for all $t>K_1$,
\begin{equation}\label{DC-BV}
\| \cL^{\lceil c \ln t \rceil}_t u \|_t \leq e^{-C\lceil c \ln t \rceil}\|u\|_t
\end{equation}
where $\|h\|_t = (1+t)^{-1}\|h\|_{\text{BV}}+\| h\|_{\text{L}^1}$. Therefore, 
$$\| \cL^{k\lceil c \ln t \rceil}_t u \|_t\leq e^{-C\lceil c \ln t \rceil}\|\cL^{(k-1)\lceil c \ln t \rceil}u\|_t \leq  \dots \leq e^{-Ck\lceil c \ln t \rceil} \|u \|_t$$ 
Also, $\|\cL_t\|_t \leq 1$. So, if $n=k\lceil c \ln t \rceil+r$ where $0\leq r< \lceil c \ln t \rceil $ then
$$\| \cL^{n}_t u \|_t \leq e^{-C k \lceil c \ln t \rceil} \|\cL^{r}_t u \|_t  \leq e^{-Cn\frac{k \lceil c \ln t \rceil}{k \lceil c \ln t \rceil + r }} \|u\|_t \leq e^{-Cn\frac{k }{k +1 }} \|u\|_t $$ 
However, 
$$(1+t)^{-1}\|h\|_{\text{BV}}\leq \|h\|_t \leq [1+(1+t)^{-1}]\|h\|_{\text{BV}}$$
Therefore, 
$$(1+t)^{-1}\| \cL^{n}_t u \|_{\text{BV}} \leq [1+(1+t)^{-1}]e^{-Cn\frac{k }{k +1 }} \|u\|_{\text{BV}}$$
which gives us
$$\|\cL^{n}_t \|_{\text{BV}} \leq (t+2)e^{-Cn\frac{k }{k +1}}$$
and here $k=k(n,t)= \lfloor \frac{n}{\lceil c \ln t \rceil}\rfloor$.
When $K_1 \leq |t| \leq n^{r_1}$, $k_{\min} = \lfloor \frac{n}{\lceil c \ln n^{r_1} \rceil}\rfloor$ and $\frac{k_{\min}}{k_{\min}+1} \to 1$ as $n \to \infty$. Also, $1 \geq \frac{k}{k+1} \geq \frac{k_{\min}}{k_{\min}+1}$ and,
$$\|\cL^{n}_t \|_{\text{BV}} \leq (t+2)e^{-Cn\frac{k(n,t) }{k(n,t) +1}} \leq 2n^{r_1}e^{-Cn\frac{k_{\min}}{k_{\min}+1} }$$
Choosing $n_0$ such that for all $n>n_0$, $\frac{k_{\min}}{k_{\min}+1}>\frac{1}{2}$ (so this choice of $n_0$ works for all $t$) we can conclude that,
$$\|\cL^{n}_t \|_{\text{BV}} \leq 2n^{r_1}e^{-Cn/2}$$
This proves (A4) for all choices of $r_1$. In particular given $r$, we can choose $r_1>\frac{r-1}{2}$ in the above proof. This implies that Edgeworth expansions of all orders exist. 


\subsection{Multidimensional expanding maps}
Let $\cM$ be a compact Riemannian manifold and $f: \cM\to \cM$ be a $C^2$ expanding map.
Let $g:\cM\to\mathbb{R}$ be a $C^2$ function which is non homologous to constant.
The proof of Lemma 3.13 in \cite{D02} shows that this condition is equivalent to $g$ not being 
{\em infinitesimally integrable} in the following sense. The natural extension of $f$ acts on the
space of pairs $(\{y_n\}_{n\in\mathbb{N}}, x)$ where 
$ f(y_{n+1})=y_n$ for $n>0$ and $fy_1=x.$ Given such pair let
\begin{align*}
\Gamma(\{y_n\}, x)=\lim_{n\to\infty} \frac{\partial}{\partial x} \left[\sum_{k=0}^{n-1} g(f^k y_n)\right]
=\lim_{n\to\infty} \frac{\partial}{\partial x} \left[\sum_{k=1}^{n} g(y_k)\right]=\left[\sum_{k=1}^{\infty} \frac{\partial}{\partial x} g(y_k)\right]. 
\end{align*}
$g$ is called infinitesimally integrable if $\Gamma(\{y_n\}, x)$ actually depends only on $x$ but not on $\{y_n\}.$ 

Let $X_n=g\circ f^n.$ We want to verify (A1)--(A4) when $x$ is distributed according to a smooth
density $\rho.$ Note that assumption \eqref{MainAssum} holds with
$v=\rho,$ $\ell$ being the Lebesgue measure and 
$$ (\cL_t\phi)(x)=\sum_{y\in f^{-1} (x)} \frac{e^{i t g(y)}}{\left|\det\left(\frac{\partial f}{\partial x}\right)\right|}
\phi(y). $$
We will check (A1)--(A4) for $\cL_t$ acting on $C^1(\cM).$ The proof of (A1)--(A3) is the same as in \cref{EM1d}. In particular, for (A3) we need Lasota--Yorke inequality (see \eqref{LYEM} below) which is proven in \cite[equation (19)]{D02}.

The proof of (A4) is also similar to \cref{EM1d}, so we just explain the differences. As before we assume that $t>0.$ Given a small constant $\kappa$ let
$$ \|\phi\|_t=\max\left(\|\phi\|_{C^0}, \frac{\kappa \|D\phi\|_{C^0}}{1+t}\right). $$
Then by \cite[Proposition 3.16]{D02}
\begin{equation}\label{A}
\|\cL_t^n \phi\|_t\leq \|\phi\|_t 
\end{equation}
provided that $n\geq C_1\ln t.$

By \cite[Lemma 3.18]{D02} if $g$ is not infinitesimally integrable then there exists a constant
$\eta<1$ such that 
\begin{equation}
\|\cL_t^n \phi\|_{L^1} \leq \eta^n \|\phi\|_t .
\end{equation}
The Lasota--Yorke inequality says that there is a constant $\theta<1,$ such that 
\begin{equation}
\label{LYEM} 
\left\Vert D\left(\cL_t^n \phi\right)\right\Vert_{C^0} \leq C_3 \left(t \|\phi\|_{C^0}+\theta^n \|D \phi\|_{C^0} \right)
\end{equation}
Also, 
\begin{equation}\label{D}
\left\Vert \cL_t^n \phi \right\Vert_{C^0} \leq
\|\cL_0^n  (|\phi|)\|_{C^0} \leq
C_4 \left(\|\; |\phi|\; \|_{L^1}+\theta^n \|\; |\phi|\; \|_{\text{Lip}}\right)
\end{equation}
where the last step relies on $\cL_0$ having a spectral gap on the space of Lipshitz functions.
Combing \eqref{A} through \eqref{D}, we conclude that $\cL_t$ satisfies \eqref{DC-BV}. The rest of the argument is the same as in \cref{EM1d}.

\appendix 

\section{} \label{appen}

In the discussion below, we do not assume the abstract setting introduced in \cref{results}. Therefore the hierarchy of asymptotic expansions provided here holds true in general.  

We observe that the classical Edgeworth expansion is the strongest form of asymptotic expansion among the expansions for non-lattice random variables. The following proposition and \cref{StrongWeak} establish this fact. 

\begin{prop}\label{Hierarchy1}
Suppose $S_N$ admits order $r$ Edgeworth expansions, then it also admits order $r$ weak global expansion for $f\in F^1_{0}$ and order $r$ averaged expansions for $f \in L^1$. Further, if the polynomials $P_{p}$ in the Edgeworth expansion has opposite parity as $p$ then $S_N$ admits  order $r-1$ weak local expansion for $f \in F^1_{r}$.
\end{prop}
\begin{rem}\label{StrongWeak}
\Cref{FiniteMC} contains examples for which the weak and averaged forms of expansions exist but the strong expansion does not. Therefore none of the above implications are reversible. 
\end{rem}
\begin{proof}[Proof of Proposition \ref{Hierarchy1}]\
Suppose $f \in F^1_{0}$. Let $F_n=\Prob\big(\frac{S_n-nA}{\sqrt{n}}\leq x \big)$ and put
$$\cE_{r,n}(x) = \fN(x)+\sum_{p=1}^r \frac{1}{n^{p/2}} P_p(x)\fn(x).$$
Observe that $F_n(x)-\cE_n(x)=o(n^{-r/2})$ uniformly in $x$ and,
\begin{align*}
d\cE_{r,n}(x) &= \fn(x)\, dx + \sum_{p=1}^r \frac{1}{n^{p/2}} \left[P'_p\left(x\right)\fn\left(x\right) + P_p(x)\fn'(x)\right]\, dx =\sum_{p=0}^r \frac{1}{n^{p/2}} R_p(x)\fn(x)\, dx
\end{align*}
where $R_p$ are polynomials given by $R_p=P'_p+P_pQ$ and $Q$ is such that $\fn'(x)=Q(x) \fn(x)$. 
Next, we observe that, 
\begin{align*}
\EXP(f(S_n-nA))&=\EXP\Big(f\Big(\frac{S_n-nA}{\sqrt{n}}\sqrt{n}\Big)\Big)= \int f(x\sqrt{n}) \, dF_n(x)\\ &= \int f(x\sqrt{n}) \, d\cE_{r,n}(x) +\int f(x\sqrt{n})\, d(F_n -\cE_{r,n})(x).
\end{align*}

Now we integrate by parts and use $\cE_{r,n}(\infty)=F_n(\infty)=1$ and $\cE_{r,n}(-\infty)=F_n(-\infty)=0$ to obtain,

\begin{align*}
\EXP(f(S_n-nA))&=  \int f(x\sqrt{n}) \, d\cE_{r,n}(x) + (F_n-\cE_{r,n})(x)f(x\sqrt{n}) \Big|_{-\infty}^{\infty} \\ &\ \hspace{180pt}  - \int (F_n  -\cE_{r,n})(x)\sqrt{n} f'(x\sqrt{n})\, dx \\ &= \int \sum_{p=0}^r \frac{1}{n^{p/2}}R_p(x)\fn(x)\, f(x\sqrt{n})dx+ o\left(n^{-r/2}\right)\int\sqrt{n} f'(x\sqrt{n})\, dx  \\ &=\sum_{p=0}^r \frac{1}{n^{p/2}} \int R_p(x)\fn(x)\, f(x\sqrt{n})dx+ o\left(n^{-r/2}\right).
\end{align*}

This is the order $r$ weak global Edgeworth expansion. The existence of the order $r-1$ weak local expansion follows from this. This is our next theorem. So we postpone its proof.  

For $f \in L^1$ substituting $x$ by $x+\frac{y}{\sqrt{n}}$ in the Edgeworth expansion for $S_n$ we have 
\begin{multline*}
\Prob\left(\frac{S_n-nA}{\sqrt{n}}\leq x+\frac{y}{\sqrt{n}}\right)-\fN\left(x+\frac{y}{\sqrt{n}}\right) \\ =\sum_{p=1}^r \frac{1}{n^{p/2}} P_p\left(x+\frac{y}{\sqrt{n}}\right)\fn\left(x+\frac{y}{\sqrt{n}}\right)+ o\left(n^{-r/2}\right).
\end{multline*}
For fixed $x$, the error is uniform in $y$. Therefore, multiplying the equation by $f(y)$ and then integrating we can conclude that the order $r$ averaged expansion exists.
\end{proof}

\begin{rem}
We have seen from the derivation of the Edgeworth expansion in \cref{proofs} that $P_p(x)$ and $p$ have opposite parity in the weakly dependent case. This implies that $P_{p,g}$ has the same parity as $p$. This is true in the i.i.d.\hspace{3pt}case as well. Even though this assumption may look artificial in the general case, it is reasonable. When using characteristic functions to derive the expansions, one is likely to end up with Hermite polynomials which is the reason behind the parity relation.
\end{rem}

Next, we compare the the relationships among the weak and averaged forms of Edgeworth expansions. 

\begin{prop}
Suppose $S_N$ admits order $r$ weak global Edgeworth expansion for $f \in F^{q+1}_r$ for some $q \geq 0$. 
If the polynomials $P_{p,g}$ in the global Edgeworth expansion has the same parity as $p$ then $S_N$ admits  order $r-1$ weak local expansion for $f$.
\end{prop}

\begin{proof}
Assume, $f\in F^1_{r}$. Then, from the Plancherel formula, 
$$\int_\reals \sqrt{n}f\big(x\sqrt{n}\big)P_{p,g}(x) \fn(x) \, dx = \frac{1}{2\pi}\int_{\reals}\widehat{f}\Big(\frac{t}{\sqrt{n}}\Big) A_p(t) e^{-\frac{\sigma^2t^2}{2}} \, dt $$
where $A_p(t)$ are polynomials constructed using the following relation, 
$$P_{p,g}(t)e^{-\frac{t^2}{2\sigma^2}} =\frac{1}{\sqrt{2\pi\sigma^2}}A_p\left(-i\frac{d}{dt}\right)\Big[e^{-\frac{t^2}{2\sigma^2}} \Big].$$ 
By construction $P_{p,g}$ and $A_p$ has the same parity. This means $A_p$ has the same parity as $p$. 

First replace $$\int P_{p,g}(x)\fn(x)\, f(x\sqrt{n})dx$$ by $$\frac{1}{2\pi \sqrt{n}}\int_{\reals}\widehat{f}\Big(\frac{t}{\sqrt{n}}\Big) A_p(t) e^{-\frac{\sigma^2t^2}{2}} \, dt$$ 
in the weak global expansion to obtain,  

\begin{align*}
\sqrt{n}\EXP(f(S_n-nA)) =\frac{1}{2\pi}\sum_{p=0}^r \frac{1}{n^{p/2}} \int_{\reals}\widehat{f}\Big(\frac{t}{\sqrt{n}}\Big) &A_p(t) e^{-\frac{\sigma^2t^2}{2}} \, dt + o\left(n^{-(r-1)/2}\right).
\end{align*}
Then substituting for $\widehat{f}$ with its order $r-1$ Taylor expansion, 
\begin{multline*}
\sqrt{n}\EXP(f(S_n-nA)) =\frac{1}{2\pi}\sum_{p=0}^r\sum_{j=0}^{r-1} \frac{\widehat{f}^{(j)}(0)}{j!n^{(j+p)/2}} \int_\reals t^j e^{-\sigma^2t^2/2} A_p(t)\, dt +o\left(n^{-(r-1)/2}\right).
\end{multline*}
Put
$$a_{pj}=\int_\reals t^j e^{-\sigma^2t^2/2} A_p(t)\, dt =0\ \ \text{and}\ \ f^{(j)}(0)=\int (-it)^jf(t) \, dt$$
to get,
\begin{align*}
\sqrt{n}\EXP(f(S_n-nA)) =\frac{1}{2\pi}\sum_{p=0}^r\sum_{j=0}^{r-1} \frac{a_{pj}}{j!n^{(j+p)/2}} \int_\reals (-it)^j f(t)\, dt + o\left(n^{-(r-1)/2}\right)
\end{align*}
Since $p$ and $A_p$ are of the same parity, when $j+p$ is odd. $a_{pj}=0$. So we collect terms such that $p+j=2k$ where $k=0,\dots,r-1$ and write,
$$P_{k,w}=\sum_{p+j=2k}\frac{a_{pj}}{j!}(-it)^j$$
Then, rearranging, simplifying and absorbing higher order terms to the error, we obtain, 
\begin{align*}
\sqrt{n}\EXP(f(S_n-nA)) =\frac{1}{2\pi}\sum_{k=0}^{\lfloor (r-1)/2\rfloor}\frac{1}{n^k}\int_\reals P_{k,w}(t)f(t)\, dt+ o\left(n^{-(r-1)/2}\right)
\end{align*}
which is the order $r-1$ weak local Edgeworth expansion. 
\end{proof}

\subsection*{Acknowledgement} The authors would like to thank Dmitry Dolgopyat for useful discussions and suggestions during the project and carefully reading the manuscript. 


\begin{thebibliography}{999}
\bibitem{Bi}  Birkhoff, Garrett; {\it Extensions of Jentzsch’s theorem}. Trans. Amer. Math. Soc. {\bf 85} (1957), no. 1, 219--227.
\bibitem{Br}  Breuillard, Emmanuel; {\it Distributions diophantiennes et theoreme limite local sur} 
$\reals^d.$  Probab. Theory Related Fields {\bf 132} (2005), no. 1, 39--73.
\bibitem{BE} Butterley, Oliver; Eslami, Peyman; {\it Exponential mixing for skew products with discontinuities}. Trans. Amer. Math. Soc. {\bf 369} (2017), no. 2, 783--803.
\bibitem{CP} Coelho, Zaqueu; Parry, William; {\it Central limit asymptotics for shifts of finite type}. Israel J. Math. {\bf 69}, (1990), no. 2, 235--249. 
\bibitem{D02} Dolgopyat, Dmitry; {\it On mixing properties of compact group extensions of hyperbolic systems}. Israel J. Math. {\bf 130} (2002), 157--205. 
\bibitem{DF}  Dolgopyat, Dmitry; Fernando, Kasun; {\it An error term in the Central Limit Theorem for sums of discrete random variables}. preprint.
\bibitem{ES} Ess\'{e}en, Carl--Gustav; {\it Fourier analysis of distribution functions.
A mathematical study of the Laplace-Gaussian law,} Acta Math. {\bf 77} (1945) 1--125. 
\bibitem {Feller2} Feller, William, An introduction to probability theory and its applications {V}ol. {II}., {Second edition}, {John Wiley \& Sons, Inc., New York-London-Sydney}, {1971}, {xxiv+669}.
\bibitem{GH}  G\"otze, Friedrich; Hipp, Christian; {\it Asymptotic Expansions for sums of Weakly Dependent Random Vectors}, Z. Wahrscheinlickeitstheorie verw., {\bf 64} (1983) 211-239.
\bibitem{G}  Gouezel, Sebastien; {\it Limit theorems in dynamical systems using the spectral method. Hyperbolic dynamics, fluctuations and large deviations}, 
Proc. Sympos. Pure Math., 89 (2015) 161--193, AMS, Providence, RI. 
\bibitem{Ha} Hall, Peter; {\it Contributions of Rabi Bhattacharya to the Central
Limit Theory and Normal Approximation.} In Manfred Denker \& Edward C. Waymire (Eds.), Rabi N. Bhattacharya Selected Papers, (pp 3--13).  Birkh\"{a}user Basel, 2016.  
\bibitem{HH} Hennion, Hubert; Herv\'e, Lo\"ic;  {Limit Theorems for Markov Chains and Stochastic Properties of Dynamical Systems by Quasi-Compactness}, {Lecture Notes in Mathematics}, {first edition}, {Springer-Verlag}, {Berlin Heidelberg}, {2001}, {viii+125}.
\bibitem{HP}  Herv\'e, Lo\"ic; P\`ene, Fran\c{c}oise; {\it The Nagaev-Guivarc\textsc{\char13}h method via the Keller-Liverani theorem}, Bull. Soc. Math. France {\bf 138} (2010) no. 3, 415--489. 
\bibitem{IL}  Ibragimov, Il'dar Abdullovich; Linnik, Yurii Vladimirovich; 
Independent and stationary sequences of random variables. With a supplementary chapter by I. A. Ibragimov and V. V. Petrov. Translation from the Russian edited by J. F. C. Kingman. Wolters-Noordhoff Publishing, Groningen, 1971. 443 pp. 
\bibitem{Kato} Kato, Tosio; {Perturbation theory for linear operators}, {Classics in Mathematics}, {Reprint of the 1980 edition}, {Springer-Verlag}, {Berlin}, {1995}, {xxii+619}.
\bibitem{L1} Liverani, Carlangelo; Decay of correlations for piecewise expanding maps. J. Statist. Phys. {\bf 78} (1995), no. 3-4, 1111--1129. 
\bibitem{NG1}  Nagaev, Sergey V.;
{\it More Exact Statement of Limit Theorems for Homogeneous Markov Chain}, Theory Probab. Appl.,  {\bf 6(1)} (1961)  62--81.
\bibitem{NG2}   Nagaev, Sergey V.;
{\it Some Limit Theorems for Stationary Markov Chains}, Theory Probab. Appl.,  {\bf 2(4)} (1959)  378--406. 
\bibitem{PN} P\`ene, Fran\c{c}oise; {\it Mixing and decorrelation in infinite measure: the case of periodic Sinai billiard}, {\tt arXiv:1706.04461v1 [math.DS]} (2017) (preprint).
\bibitem{RS} Rubin, Herman; Sethuraman, Jayaram; {\it Probabilities of moderate deviations},
{Sankhya Ser. A}, {\bf 27} (1965) 325--346.
\bibitem{RS1} Rubin, Herman; Sethuraman, Jayaram; {\it Bayes risk efficiency},
{Sankhya Ser. A}, {\bf 27} (1965) 347--356.
\end{thebibliography}
\end{document}